\documentclass{EngCLS}

\title{Dynamic Programming for Value-at-Risk Markov Decision Processes in Continuous Spaces\thanks{Author names are listed alphabetically.}}

\author{Xian Chen\thanks{School of Mathematical Sciences, Xiamen University (\texttt{chenxian@xmu.edu.cn})}
\quad Erick Delage\thanks{GERAD \& Department of Decision Sciences, HEC Montréal (\texttt{erick.delage@hec.ca})}
\quad Weikai Wang\thanks{Corresponding author; GERAD \& Department of Decision Sciences, HEC Montréal (\texttt{weikai.wang@hec.ca})}
}

\date{\today}

\begin{document}

\maketitle

\begin{abstract}
    Value-at-Risk (VaR) is a widely adopted quantile-based risk measure in practice, yet incorporating VaR objectives into Markov decision processes (VaR MDPs) remains a longstanding challenge.
    Recent work has derived dynamic programming equations for VaR MDPs in finite state and action spaces but left the case of more general spaces unresolved. This paper establishes, for the first time, the validity of these equations in continuous spaces, for both finite- and infinite-horizon settings. Such cases naturally arise in many real-world dynamic decision-making problems. As an application, we study an optimal execution problem with a VaR objective and demonstrate that our framework enables the characterization of optimal dynamic strategies, resolving an open question in the field.
\end{abstract}

\section{Introduction}

Value-at-Risk (VaR), as a widely used quantile-based risk measure, is one of the most common tools for risk assessment in practice \citep{deCastro2019econometrica, deCastro2022geb, HeXD2025mnsc}. According to a field experiment \citep{deCastro2022geb}, more than 40\% of examinees are quantile risk maximizers. VaR is also introduced as a standard for determining capital risk requirements in the Basel Accord \citep{BCBS2016}. Although VaR has been criticized for not being subadditive, thus failing to qualify as a convex risk measure \citep{Artzner2001mf}, and for neglecting the magnitude of extreme tail losses, it possesses several appealing properties. For instance, VaR is comonotonic additive, elicitable, and robust with respect to weak topology and the Wasserstein distance, giving it advantages in backtesting and robust statistical analysis compared to other coherent risk measures such as Conditional Value-at-Risk (CVaR) \citep{Emmer2015jor}.

The use of VaR as objective functions has a long-standing tradition in both control theory and finance. While numerous studies adopt continuous-time stochastic models for analytical tractability, these models often fall short in accurately representing real-world data. In practice, financial observations occur in discrete time, and discrete-time stochastic processes are typically the case of incomplete markets and are more suitable for numerical computation \citep{Follmer2016book}. Consequently, it is essential to develop VaR-based optimal control methods within the framework of discrete-time Markov decision processes (MDPs).

The study of VaR MDPs dates back to \cite{Filar1995tac}, which considered only the average-reward setting. The finite- and infinite-horizon discounted cases remained unresolved until \cite{LiXC2022opre}, where the problem was studied under the name of quantile MDPs. By exploiting a VaR decomposition result and augmenting the state space with the risk level, the authors derived a dynamic programming formulation in which the optimal risk level is updated recursively along the state trajectory. However, their quantile definition gave rise to several technical issues. These were subsequently addressed in \cite{Hau2023neurips}, which established a rigorous VaR decomposition theorem and derived a correct finite-horizon VaR dynamic programming formulation.
Building on this foundation, \cite{Hau2025aistats} further derives a nested-VaR formulation using the quantile representation and establishes its equivalence to the original VaR dynamic programming recursion. This reformulation also enables the development of an off-policy Q-learning algorithm. The nested-VaR framework is further exploited by \cite{LuoY2026icml} to facilitate CVaR gradient computation and improve the efficiency of CVaR policy optimization. From a different perspective, \cite{Xia2025arXiv} exploits the duality between VaR maximization and probabilistic minimization to develop a policy iteration algorithm that recursively solves the associated probabilistic minimization MDP and recovers an optimal policy for the VaR MDP.

However, existing studies on VaR MDPs are all restricted to finite state spaces, and thus do not provide a theoretical foundation for their application to problems, such as those in finance, where states or actions are naturally continuous.
Extending VaR dynamic programming to continuous spaces therefore remains an important open challenge, as highlighted in several studies \citep{JinYS2017qf, Hau2025aistats, HeXD2025mnsc}. This work aims to address this gap.
 
\subsection{Our Contributions}
We summarize our contributions as follows:
\begin{itemize}
    \item We establish a VaR decomposition theorem for random variables on Borel spaces,
    extending the results of \cite{LiXC2022opre, Hau2023neurips} from discrete outcome spaces to more general Borel spaces.
    
    \item Building on the VaR decomposition theorem, we develop dynamic programming formulations for finite- and infinite-horizon VaR MDPs on Borel state and action spaces by augmenting the state with the risk level. Extending the approaches of \cite{LiXC2022opre} and \cite{Hau2025aistats}, we define two sets of Bellman optimality equations, referred to as the VaR and nested-VaR dynamic programming formulations, prove their equivalence, and show that both recover the optimal values of the VaR MDP objectives. An optimal policy is then characterized through optimally updated risk levels.
    
    \item As an application, we study the continuous-state optimal execution problem under a VaR objective, extending the static formulation of \cite{Almgren2001jor} to dynamic trading strategies. We identify two regimes determined by the relative magnitudes of permanent and temporary market impact. In the first regime, the problem reduces to optimal stopping, and we derive an optimal dynamic strategy that strictly outperforms the optimal static strategy over an explicitly characterized range of risk levels. The policy admits a simple threshold implementation, eliminating stage-wise risk-level updates. In the second regime, we derive recursive equations for the optimal value and dynamic strategy. Numerical methods are proposed for computing the optimal strategies for both regimes, and experimental results demonstrate both the potential gains from dynamic execution and the applicability of the proposed VaR dynamic programming framework.
\end{itemize}

The remainder of the paper is organized as follows. Section~\ref{sec:Preliminaries} introduces the basic notation for MDPs on Borel spaces and formulates the VaR MDP problem. Section~\ref{sec:VaRDecomposition} establishes the VaR decomposition theorem and supporting technical results. Sections~\ref{sec:VaRDPFinite} and \ref{sec:VaRDPInfinite} develop the VaR dynamic programming framework for the finite- and infinite-horizon settings, respectively. Section~\ref{sec:VaRExecution} applies the framework to an optimal execution problem and presents numerical results. Finally, Section~\ref{sec:Conclusion} summarizes the paper and proposes further research.

\subsection{Related Literature}

Beyond recent work on dynamic programming for VaR in finite spaces \citep{LiXC2022opre, Hau2023neurips, Hau2025aistats, Xia2025arXiv}, our study is closely related to risk-sensitive MDPs with static risk measures, where the risk measure is applied directly to cumulative rewards. Since the seminal work of Howard and Matheson \cite{Howard1972mnsc}, this literature has focused primarily on the entropic risk measure (ERM) under the average-reward criterion; see \cite{Biswas2023arc, Bauerle2024mmor} for detailed reviews. For more general static risk measures, Bellman recursions are typically unavailable because of time inconsistency \citep{Rudloff2014ejor}, stemming from the absence of a tower property analogous to that of expectation and ERM. This difficulty can be addressed by augmenting the state with sufficient statistics of the history, an approach originated in MDPs with expected-utility objectives \citep{Kreps1977mor}, thereby restoring time consistency in the augmented MDP. This technique has been applied to MDPs with various risk objectives, such as CVaR \citep{Bauerle2011mmor}, certainty equivalent  \citep{Bauerle2014mor}, spectral risk measures \citep{Bauerle2021mmor}, and optimized certainty equivalent \cite{Bauerle2024mmor}.  
More recently, ERM and entropic VaR have been studied under finite- and infinite-horizon discounted criteria without state augmentation, exploiting its particular structural properties \citep{Hau2023aistats}. In contrast, VaR generally falls outside the risk measures that have been considered in these studies due to its non-convexity.

Another line of research addresses time inconsistency through recursive risk evaluation.
\cite{Ruszczynski2010mp} develops this framework for finite- and infinite-horizon discounted MDPs with dynamic coherent risk measures, with subsequent extensions to general convex risk measures \citep{Shen2013sicon}, unbounded costs \citep{Bauerle2022ejor}, and transient MDPs \citep{Cavus2014sicon}. \cite{Fan2022mp} provides a related foundation by deriving dynamic programming equations for process-based risk measures under stochastic conditional time consistency. Closely related to our work, \cite{Jiang2018mor} develops simulation-based approximate dynamic programming for dynamic quantile-based risk measures in finite-horizon MDPs, while \cite{deCastro2019econometrica} studies dynamic VaR over an infinite horizon and derives an Euler equation characterizing optimal decisions. These recursive risk criteria generally differ from applying a single static risk measure to cumulative rewards, which is the VaR objective considered in this work.

Our VaR decomposition theorem complements existing risk decomposition results. \cite{Pflug2016mor} develops a decomposition for coherent risk measures, while \cite{Dentcheva2020mp} establishes a disintegration formula for conditionally consistent risk forms on Polish spaces. 
VaR generally satisfies neither coherence nor conditional consistency, 
yet our decomposition successfully enables policy optimization of such a measure on general Borel spaces using  dynamic programming.

From an application perspective, our work contributes to VaR-based optimal execution. The problem was first introduced in the seminal work of \cite{Almgren2001jor}, which derived optimal static strategies from the mean–variance efficient frontier but left optimal dynamic strategies unresolved. \cite{JinYS2017qf} studies static execution through loss-probability minimization, computes the associated liquidity-adjusted VaR, and identifies adaptive execution as a direction for further research. For a CVaR objective, \cite{MinSK2022arXiv} derives optimal adaptive strategies in a continuous-time, infinite-horizon execution model. 
To the best of our knowledge, the characterization of optimal dynamic strategies for VaR-based execution remains an open problem. Section \ref{sec:VaRExecution} addresses this gap by applying the proposed VaR dynamic programming framework to maximize the static VaR of execution revenue in a discrete-time finite-horizon model, to characterize optimal dynamic strategies, and to evaluate their performance against static benchmarks.

\section{Preliminaries}\label{sec:Preliminaries}

\textbf{Notations:} Let $E$ be a Borel space, i.e., a Borel subset of a complete separable metric space, and let $\mcB(E)$ denote its Borel $\sigma$-algebra. 
Throughout, measurability between Borel spaces is understood with respect to their Borel $\sigma$-algebras.
Denote by $\mcL(E)$ and $\mcC_b(E)$ the spaces of real-valued measurable functions and bounded continuous functions on $E$, respectively. 
Let $\bar{\mbR} := \mbR \cup\{-\infty,+\infty\}$ be the extended real line, with the conventions $\sup\{\varnothing\} := -\infty$, $\inf\{\varnothing\} := +\infty$. All operations involving $\pm\infty$ are understood in $\bar{\mbR}$ whenever well-defined.
Let $\mbN$, $\mbN_+$, $\mbQ$ denote the sets of non-negative integers, positive integers, and rational numbers, respectively. We use $E^{\mbN}$ to denote the countable infinite products of $E$.
For $T \in \mbN$, write $[T] := \{0,\cdots,T\}$. Denote $\mcP(E)$ the set of all probability measures on $(E,\mcB(E))$.  For $x\in E$, let $\delta_x$ denote the Dirac measure at $x$; for $A \in \mcB(E)$, let $\bm{1}_{A}$ denote the indicator function. Denote by $\mathrm{Unif}[0,1]$ the uniform distribution on $[0,1]$, and by $\mcN(a,b^2)$ the normal distribution with mean $a$ and variance $b^2$.

In this section, we first introduce the notation for MDPs in general Borel spaces. We then present the definition of the VaR and VaR MDPs for both finite and infinite horizons, along with their corresponding optimization objectives.

\subsection{MDPs in General Space}
A Markov decision process (MDP) is specified by the tuple 
\[ (\bfX, \bfA, \{\bfA(x)\subset \bfA: x \in \bfX\}, Q, r_{t}). \]
The state space $\bfX$ and the action space $\bfA$ are standard Borel spaces endowed with their Borel $\sigma$-algebras $\mcB(\bfX)$ and $\mcB(\bfA)$, respectively. For each $x \in \bfX$, let $\bfA(x) \subseteq \bfA$ denote the nonempty set of admissible actions, and define the feasible state-action space by $\bfK := \{ (x,a): x\in \bfX, a\in \bfA(x)\}$, which is assumed to belong to $\mcB(\bfX)\otimes\mcB(\bfA)$. The system dynamics are governed by a stochastic kernel $Q(\cdot | \cdot)$ on $\bfX$ given $\bfK$; that is, $Q(\cdot | x,a)$ is a probability measure on $(\bfX,\mcB(\bfX))$ for every $(x,a) \in \bfK$, and $Q(C | \cdot)$ is measurable on $\bfK$ for every $C \in \mcB(\bfX)$.
At each time $t \ge 0$, when choosing action $a\in \bfA(x)$ at state $x$, a one-stage reward $r_t(x,a) \in \mbR$ is generated. 

Let $\bfH_0 := \bfX$ and for $t \in \mbN_+$, define $\bfH_t:= \bfK^t\times \bfX$. An element $h_t = (x_0,a_0,\cdots,x_t) \in \bfH_t$ is referred to as a history up to time $t$. Each $\bfH_t$ is endowed with its Borel $\sigma$-algebra.
Let $\mbF$ denote the set of all measurable selectors of $\bfA$, i.e., $\mbF := \{f: \bfX\to \bfA: f \text{ is measurable and } f(x) \in \bfA(x), \forall x\in \bfX \}$.

For $T\in\mbN_+$, a finite horizon admissible history-dependent policy is a sequence $\pi_T = (\varpi_0,\cdots,\varpi_{T-1})$, where for each $t\in[T-1]$, $\varpi_t$ is a  stochastic kernels on $\bfA$ given $\bfH_t$ satisfying $\varpi_t(A(x_t) | h_t) = 1$, for $h_t = (x_0,a_0,\cdots,x_t) \in \bfH_t$. The policy is called deterministic if there exists measurable selectors $f_t: \bfH_t \to \bfA$, such that $f_t(h_t) \in \bfA(x_t)$ , $\varpi_t(\cdot|h_t) = \delta_{f_t(h_t)}(\cdot)$, and Markov, if there exists $f_t \in \mbF$ such that $f_t(x_t)\in\bfA(x_t)$, $\varpi_t(\cdot|h_t) = \delta_{f_t(x_t)}(\cdot)$ for all $t\in[T-1]$. 
Let $\Pi_T$ denote the class of all admissible history-dependent policies over $T$ periods, and set $\Pi_0 = \{\varnothing\}$, where $\varnothing$ denotes the empty policy. An infinite horizon admissible history-dependent policy is a sequence $\pi = (\varpi_0,\varpi_1,\cdots)$, and the corresponding infinite-horizon policy class is denoted by $\Pi$.

For $T \in \mbN$, let $\Omega^T := \bfK^T\times\bfX$ be equipped with its product Borel $\sigma$-algebra $\mcB(\Omega^T)$. Given an initial distribution $\nu \in \mcP(\bfX)$ and a policy $\pi_T = (\varpi_0,\cdots,\varpi_{T-1})\in\Pi_T$, the Ionescu-Tulcea theorem (Proposition C.10 in \cite{Hernandez1996book}) implies the existence of a unique probability measure $\mbP_\nu^{\pi,T}$ on $(\Omega^T,\mcB(\Omega^T))$ given by
\begin{equation}\label{eq-IonescuTulceaTheorem}
    \mbP_\nu^{\pi_T,T}(dx_0,da_0,\cdots,dx_T) = \nu(dx_0) \prod_{t=0}^{T-1} (\varpi_t(da_t|(x_0,a_0,\cdots,x_t)) Q(dx_{t+1}|x_t,a_t))
\end{equation}
Let $\mbE_\nu^{\pi_T,T}$ denote the expectation w.r.t. $\mbP_\nu^{\pi_T,T}$. For a fixed initial state $x\in\bfX$, we write $\mbP_x^{\pi_T,T}:=\mbP_{\delta_x}^{\pi_T,T}$ and $\mbE_x^{\pi_T,T} := \mbE_{\delta_x}^{\pi_T,T}$, respectively.
When $T = 0$, the empty-product convention gives $\Omega^0 = \bfX$, $\mbP_\nu^{\varnothing,0} = \nu$ and $\mbP_x^{\varnothing,0} = \delta_x$. 

For the infinite-horizon model, let $\Omega^\infty := \bfK^{\mbN}$ be equipped with its product Borel $\sigma$-algebra $\mcB(\Omega^\infty)$. Given $\nu \in \mcP(\bfX)$ and $\pi = (\varpi_0,\varpi_1,\cdots) \in \Pi$, the Ionescu-Tulcea theorem yields a unique probability measure $\mbP_\nu^\pi$ on $(\Omega^\infty, \mcB(\Omega^\infty))$ whose finite-dimensional marginals satisfy
\[
    \mbP_\nu^\pi \circ \mathrm{proj}_T^{-1} = \mbP_\nu^{\pi_{0:T-1},T},\quad \forall T\in\mbN_+,
\]
where $\mathrm{proj}_T:\Omega^\infty \to \Omega^T$ is the canonical projection, i.e., for $\omega = (x_0,a_0,x_1,a_1,\cdots)\in\Omega^\infty$, we have $\mathrm{proj}_T(\omega) = (x_0,a_0,\cdots,x_T)$; and $\pi_{0:T-1} := (\varpi_0,\cdots,\varpi_{T-1}) \in \Pi_T$. Let $\mbE_\nu^\pi$ denote the expectation to $\mbP_\nu^\pi$. For $x\in\bfX$, we write $\mbP_x^\pi := \mbP_{\delta_x}^\pi$, $\mbE_x^\pi := \mbE_{\delta_x}^\pi$, respectively.

\subsection{VaR MDPs}
Let $(\Omega,\mcF,\mbP)$ be a standard probability space where $\mcF$ is a $\sigma$-algebra on $\Omega$ and $\mbP$ is a probability measure. Following \cite{Follmer2016book}, for a random variable $Z \in \mcL(\Omega)$, we define its \textit{Value-at-Risk at level $\alpha$} ($\alpha$-VaR) by
\begin{equation}\label{eq-ValueatRisk}
    \VaR_\alpha[Z] := \sup\{ z \in \mbR: \mbP[Z < z] \le \alpha\}, \quad \alpha \in [0,1),
\end{equation}
and set $\VaR_1[Z] := +\infty$. Since $z\mapsto \mbP[Z < z]$ is left-continuous, the supremum is attained whenever it is finite. Moreover, $\VaR_0[Z] = \ess\inf Z$, with the convention $\VaR_0[Z] := -\infty$ when $Z$ is unbounded below.

Consider the MDP introduced in the previous section. Fix a horizon $T\in\mbN$, a discount factor $\gamma \in (0,1]$, and a risk level $\alpha \in [0,1)$. Let $\bar{\lambda}$ be the uniform probability measure on $([0,1],\mcB([0,1]))$, and define
\[
\bar{\mbP}_x^{\pi_T,T} := \mbP_x^{\pi_T,T} \otimes\bar{\lambda},\quad \forall x\in\bfX, \pi_T\in\Pi_T.
\]
Let $U_T\sim \mathrm{Unif}[0,1]$ and be independent of the state-action process. Given a measurable terminal reward function $g:\bfX\times[0,1] \to \bar{\mbR}$, define
\begin{equation}\label{eq-objectivefunction-FH}
    J_{\alpha,T}(x,\pi_T) := \VaR_{\alpha}^{\bar{\mbP}_x^{\pi_T,T}} \left[ \sum_{t=0}^{T-1} \gamma^t r_t(X_t,A_t) + \gamma^T g(X_T,U_T) \right],\quad \forall x\in\bfX,\pi_T\in\Pi_T.
\end{equation}
The optimal VaR discounted cumulated reward objective function is denoted by
\begin{equation}\label{eq-OptimalObjective-FH}
    J_{\alpha,T}^*(x) := \sup_{\pi_T\in\Pi_T} J_{\alpha,T}(x,\pi_T),\quad \forall x\in \bfX.
\end{equation}
The \textit{finite-horizon VaR MDP problem} is to identify an optimal policy $\pi_T^* \in \Pi_T$ such that $J_{\alpha,T}(x,\pi_T^*) = J_{\alpha,T}^*(x)$ for all $x\in \bfX$.

For the infinite-horizon setting, let $\gamma \in (0,1)$ and assume stationary rewards, i.e., $r_t \equiv r$ for some $r: \bfK \to \mbR$. The corresponding VaR discounted cumulated reward objective is
\begin{equation}\label{eq-ObjectiveFunction-IH}
    J_{\alpha,\infty}(x,\pi) := \VaR_{\alpha}^{\mbP_x^{\pi}} \left[ \sum_{t=0}^\infty \gamma^t r(X_t,A_t) \right], \quad x\in \bfX, \pi \in \Pi.
\end{equation}
Similarly, define
\begin{equation}\label{eq-optimalfunction-IH}
    J_{\alpha,\infty}^*(x) := \sup_{\pi\in\Pi} J_{\alpha,\infty}(x,\pi),\quad \forall x\in \bfX.
\end{equation}
The \textit{infinite-horizon discounted VaR MDP problem} is to find an optimal policy $\pi^* \in \Pi$ such that $J_{\alpha,\infty}^*(x) = J_{\alpha,\infty}(x,\pi^*)$ for all $x\in \bfX$. 

\begin{remark}
    It is worth noting that the VaR representation in \eqref{eq-ValueatRisk} also suggests a two-stage formulation, analogous to \cite{Bauerle2011mmor, Bauerle2021mmor, Bauerle2024mmor, Bauerle2026ejor, Xia2025arXiv}. For the finite-horizon problem \eqref{eq-objectivefunction-FH}, one may first solve, for each fixed threshold $z$, a probabilistic MDP $p_T(x,z) := \inf_{\pi_T\in\Pi_T} \bar{\mbP}_x^{\pi_T,T}[Z_T^{\pi} < z]$, for all $x\in\bfX$, and recover the optimal VaR through $\sup\{z: p_T(x,z) \le \alpha\}$, where $Z_T^{\pi_T}$ denotes the $T$-stage discounted total reward under policy $\pi_T$. An analogous formulation can be developed for the infinite-horizon problem. This threshold-based approach, however, is less tractable on general state spaces. The function $p_T(x,\cdot)$ may be discontinuous, which complicates the outer optimization. Moreover, the inner problem requires dynamic programming on an augmented state space that can be unbounded, whereas the augmentation adopted in this work uses the bounded risk level interval. 
\end{remark}

\section{VaR Decomposition Theorem}\label{sec:VaRDecomposition}
In this section, we present a general VaR decomposition theorem. The result was first introduced for finite discrete random variables in Lemma 2 of \cite{LiXC2022opre} and was subsequently refined in \cite{Hau2023neurips, Hau2025aistats}.

Let $(\Omega,\mcB(\Omega))$ and $(E,\mcB(E))$ be standard Borel spaces, and let $\mbP$ be a probability measure on $(\Omega,\mcB(\Omega))$. 
Let $Y:\Omega\to E$ be measurable, and let $\nu:=\mathbb P\circ Y^{-1}$ denote the law of $Y$. By the disintegration theorem for standard Borel spaces (see Theorem 8.5 in \cite{Kallenberg2021book}), there exists a regular conditional probability measure
\begin{equation}\label{eq-RegularConditionalMeasure}
    \mu: \mcB(\Omega)\times E \to[0,1],\qquad (B,y)\mapsto \mu(B,y),
\end{equation}
such that for each $y\in E$, $B\mapsto \mu(B,y)$ is a probability measure
on $\Omega$; for each $B\in\mathcal B(\Omega)$, the map $y\mapsto \mu(B,y)$ is $\mcB(E)$-measurable; and
\begin{equation}\label{eq-RCMIntegral}
    \int_A \mu(B,y)\nu(dy) = \mbP(B\cap Y^{-1}(A)),
\quad \forall B\in \mcB(\Omega), A\in\mcB(E).
\end{equation}
We fix an arbitrary version of such a regular conditional probability measure. Then for a random variable $X\in\mcL(\Omega)$ and $z\in\mathbb R$, we can define
\[
\mbP[X < z \mid Y = y] := \mu(\{X < z\}, y), \quad \forall y\in E.
\]
The conditional $\alpha$-VaR of $X$ given $Y = y$ is then defined as
\begin{equation}\label{eq-ConditionalVaR}
    \VaR_\alpha[X \mid Y = y] :=\sup\{z\in \mbR: \mu(\{X < z\},y) \le \alpha \},
\quad \forall \alpha\in[0,1), y\in E,
\end{equation}
and $\VaR_1[X\mid Y =y] := +\infty$, for all $y\in E$.

We denote by $\supp(\nu)$ the support of $\nu$ on $E$, i.e., the intersection of all closed subsets of $E$ whose complements have $\nu$-measure zero (see Proposition 1.45 in \cite{Follmer2016book}). 

We now give the VaR decomposition theorem for an arbitrary real-valued random variable.
\begin{thm}\label{thm:VaRDecomposition}
    Let $X$ and $Y$ be random variables defined on $(\Omega,\mcF,\mbP)$, taking values in $(\mbR,\mcB(\mbR))$ and $(E,\mcB(E))$, respectively. Let $\nu := \mbP \circ Y^{-1}$ be the law of $Y$, and let $\mcY \subset E$ satisfy $\supp(\nu) \subseteq \mcY$.
    Then for every $\alpha \in [0,1]$, we have
    \begin{equation}\label{eq-VaRDecomposition}
        \VaR_{\alpha}[X] = \sup_{\zeta\in\mcZ_{\alpha}(\nu)} \inf_{y\in\mathcal{Y}} \VaR_{\zeta(y)} [X|Y=y],
    \end{equation}
    where 
    \[
        \mcZ_\alpha(\nu) := \left\{ \zeta \in \mcL(E) : 0 \le \zeta \le 1, \int_E \zeta(y) \nu(dy) \le \alpha\right\}.
    \]
    Moreover, the supremum is attained at $\zeta_\alpha(y) := \mu(\{X < \VaR_\alpha[X]\},y)$ for $\alpha \in (0,1)$ and also for $\alpha = 0$ with $\VaR_0[X] > -\infty$, where $\mu$ is the regular conditional measure defined in \eqref{eq-RegularConditionalMeasure}. When $\alpha = 0$ with $\VaR_0[X] = -\infty$, every $\zeta \in \mcZ_0(\nu)$ attains the supremum.
\end{thm}
\begin{proof}
    When $\alpha = 1$, the supremum over $\mcZ_\alpha(\nu)$ is clearly attained by $\zeta\equiv 1$, which yields $\VaR_{\zeta(y)}[X|Y=y] = +\infty$ for all $y\in E$. So the claim follows immediately. 

    Now consider $\alpha \in (0,1)$.
    For $z \in \mbR$, define
    \[
        f_z(y) := \mu(\{X < z\},y),\quad \forall y\in E.
    \]
    Since $\mu(\cdot,y)$ is a probability measure, $0\le f_z \le 1$ and $z\mapsto f_z(y)$ is non-decreasing and left-continuous for each $y \in E$. Let $z_0 := \VaR_\alpha[X] \in \mbR$.
    By \eqref{eq-RCMIntegral}, $\int_E f_{z_0}(y) \nu(dy) = \mbP[X < \VaR_\alpha[X]] \le \alpha$,
    and hence $f_{z_0} \in \mcZ_\alpha(\nu)$.
    
    From the definition of VaR, we have
    \begin{align*}
        \VaR_\alpha[X] = \sup\{ z: \mbP[X < z] \le \alpha\}
        = \sup\left\{ z: \int_E f_z(y) \nu(dy) \le \alpha\right\}.
    \end{align*}
    Next we show that
    \begin{equation}\label{eq-VaRDecomposition-1}
        \VaR_\alpha[X] = \sup_{\zeta \in \mcZ_{\alpha}(\nu)} \sup \left\{ z: f_z(y) \le \zeta(y) ,\forall y \in \mcY \right\}.
    \end{equation}
    
    To see this, define the sets
    \begin{align*}
        D :=\left\{z: \int_E f_z(y) \nu(dy) \le \alpha \right\},\quad D_0 := \{z: f_z(y) \le f_{z_0}(y), \forall y \in E\}. 
    \end{align*}
    Clearly, for any $z \in D_0$, we have $\int_E f_z(y) \nu(dy) \le \int_E f_{z_0}(y) \nu(dy) \le \alpha$. Hence $D_0 \subset D$. On the other hand, for any $\hat{z} \in D$, if $\hat{z} > z_0$, we have $\int_E f_{z_0}(y) d\nu(y) \le \int_E f_{\hat{z}}(y) \le \alpha$ which implies that $\mbP[X < \hat{z}] \le \alpha$ and $\hat{z} > \VaR_\alpha[X]$, which contradicts the definition of $\VaR_\alpha[X]$. This implies that $z \le z_0$ for any $z\in D$. Combining with the monotonicity of $f_z$ in $z$, we conclude that $D\subset D_0$. Therefore, $D = D_0$. We obtain that for $\mcY \subseteq E$,
    \begin{align*}
        \VaR_\alpha[X] &= \sup \{ z: f_z(y) \le f_{z_0}(y), \forall y\in E\}
        \le \sup_{\zeta\in\mcZ_\alpha(\nu)} \sup \{ z: f_z(y) \le \zeta(y), \forall y\in \mcY\}.
    \end{align*}

    To show the opposite inequality, suppose there exists some $\epsilon > 0$ and $\zeta \in \mcZ_{\alpha}(\nu)$ such that $z_0 + \epsilon \in \{z : f_z(y) \le \zeta(y), \forall y\in \mcY\}$. This implies $f_{z_0 + \epsilon}(y) \le \zeta(y)$, for all $y\in\mcY$. Since $\supp(\nu)\subseteq\mcY$, integrating both sides over $y$ yields
    \begin{align*}
        \mbP[X < z_0 + \epsilon] = \int_E f_{z_0 + \epsilon}(y) \nu(dy) \le \int_E \zeta(y) \nu(dy) \le \alpha.
    \end{align*}
    By the definition of $\VaR_\alpha[X]$, this implies that $\VaR_\alpha[X] \ge \VaR_\alpha[X] + \epsilon$, which is a contradiction. Therefore, such $\epsilon > 0$ does not exist. We must have that $\{z: f_z(y) \le \zeta(y), \forall y\in \mcY\} \cap \{z: z > z_0\} = \varnothing$, for all $\zeta \in \mcZ_{\alpha}(\nu)$. This implies that 
    \[
    \VaR_\alpha[X] \ge \sup_{\zeta\in\mcZ_{\alpha}(\nu)}\sup\{z: f_z(y) \le \zeta(y), \forall y\in \mcY\}.
    \]
    Hence \eqref{eq-VaRDecomposition-1} holds and moreover, the supremum over $\mcZ_{\alpha}(\nu)$ is achieved by $\zeta_\alpha = f_{z_0}$.

    We are left with showing that for all $\zeta\in \mcZ_{\alpha}(\nu)$,
    \begin{equation}\label{eq-VaRDecomposition-2}
        \sup \left\{ z: f_z(y) \le \zeta(y), \forall y\in\mathcal{Y}\right\} =\inf_{y \in \mcY} \sup \left\{ z:  f_z(y) \le \zeta(y) \right\}.
    \end{equation}
    Since \eqref{eq-VaRDecomposition-2} combined with \eqref{eq-VaRDecomposition-1} yields
    \begin{equation}\label{eq-VaRDecomposition-3}
    \begin{aligned}
        \VaR_\alpha[X] &= \sup_{\zeta \in \mcZ_{\alpha}(\nu)}\sup \left\{ z: f_z(y) \le \zeta(y), \forall y\in\mcY \right\}\\
        &=  \sup_{\zeta \in \mcZ_{\alpha}(\nu)} \inf_{y \in \mcY} \sup \left\{ z:  f_z(y) \le \zeta(y) \right\}\\
        &= \sup_{\zeta \in \mcZ_{\alpha}(\nu)} \inf_{y \in \mcY} \VaR_{\zeta(y)}[X|Y = y],
    \end{aligned}
    \end{equation}
    where the last equality follows from \eqref{eq-ConditionalVaR}. Moreover, combining \eqref{eq-VaRDecomposition-2}, \eqref{eq-VaRDecomposition-3} and our previous analysis gives
    \begin{align*}
        \inf_{y\in \mcY} \VaR_{\zeta_\alpha(y)} [X|Y = y] &= \inf_{y\in \mcY} \sup\{z: f_z(y) \le \zeta_\alpha(y)\} \\
        &= \sup\{ z: f_z(y) \le \zeta_\alpha(y), \forall y\in \mcY\} \\
        &= \VaR_\alpha[X] = \sup_{\zeta\in\mcZ_{\alpha}(\nu)} \inf_{y\in\mcY} \VaR_{\zeta(y)}[X|Y = y],
    \end{align*}
    which implies that $\zeta_\alpha$ indeed attains the supremum in \eqref{eq-VaRDecomposition}.

    We finalize this proof by showing \eqref{eq-VaRDecomposition-2} for any fixed $\zeta\in\mcZ_{\alpha}(\nu)$. 
    For any fixed $y\in\mcY$, denote
    \[
    A_y := \{ z\in\mbR: f_z(y) \le \zeta(y)\},\quad z_y := \sup A_y \in \bar{\mbR}.
    \]
    Clearly we have $\sup\cap_{y\in\mcY} A_y \le \sup A_{y'}$ for any $y'\in\mcY$ and hence $\sup \cap_{y\in\mcY} A_y \le \inf_{y\in\mcY} z_y$. 
    Conversely, let $\bar{z} < \inf_{y\in\mcY} z_y$, then for each $y\in\mcY$, $\bar{z} < z_y$. Since $z\mapsto f_z(y)$ is non-decreasing, the set $A_y$ is downward closed, and hence $\bar{z} \in A_y$. This implies $\bar{z}\in\cap_{y\in\mcY} A_y$. Let $\bar{z}\to \inf_y z_y$, we conclude that $\inf_{y\in\mcY} z_y \le \sup \cap_{y\in\mcY} A_y$. Combined with the reversed inequality, This proves \eqref{eq-VaRDecomposition-2}.

    The case $\alpha = 0$ with $\VaR_0[X] > -\infty$ follows analogously. It remains to consider $\alpha = 0$ with $\VaR_0[X] = -\infty$. Fix an arbitrary $\zeta \in \mcZ_0(\nu)$. Since $0\le \zeta \le 1$ and $\int_E \zeta d\nu = 0$, we have $\zeta = 0$ $\nu$-almost everywhere. For any $z\in \mbR$, we have $\int_E f_z(y) \nu(dy) = \mbP[X < z] > 0$, so $D_z :=\{y: f_z(y) > 0\}$ has positive $\nu$-measure. Hence $D_z \cap \{y: \zeta(y) = 0 \} \cap \supp(\nu) \neq \varnothing$. Choose $y_z$ in this intersection. Then $\mu(\{X < z\}, y_z) > 0$, which implies $\VaR_0[X|Y = y_z] < z$. Since $\zeta(y_z) = 0$ and $\supp(\nu) \subseteq \mcY$, we have  $\inf_{y\in \mcY} \VaR_{\zeta(y)}[X|Y=y] \le \VaR_{0}[X|Y=y_z] < z$. As this holds for every $z\in\mbR$, letting $z \downarrow -\infty$ yields $\inf_{y\in\mcY} \VaR_{\zeta(y)}[X|Y=y] = -\infty$. Thus \eqref{eq-VaRDecomposition} holds for every $\zeta\in\mcZ_0(\nu)$. 
    This completes the proof.
\end{proof}

It is well known that, on an atomless probability space, any random variable can be represented as a transformation of a uniformly distributed random variable via its quantile function. The following lemma combines this quantile representation with the VaR decomposition theorem and serves as a key building block for deriving the dynamic programming formulations of the VaR MDP. 

Before presenting the lemma, we first define the function set that will be of interest. Let $E$ be a Borel space, and define $\hat{E}:= E\times[0,1]$. Denote $\mcU(\hat{E})$ the collection of all functions $\psi:\hat{E}\to\bar{\mbR}$ such that
\begin{enumerate}[label = (A\arabic*)]
        \item $x\mapsto \psi(x,\alpha)$ is measurable for each $\alpha \in [0,1]$;
        \item $\alpha \mapsto \psi(x,\alpha)$ is non-decreasing and right-continuous for each $x\in E$;
        \item $\psi(x,0) \in [-\infty, \infty)$ for all $x\in E$;
        \item $\psi(x,\alpha) \in \mbR$ for $\alpha \in (0,1)$ and $x\in E$;
        \item $\psi(x,1) = +\infty$ for $x\in E$.
    \end{enumerate}

\begin{lemma}\label{lemma:QuantileTransform}
    Let $E$ be a Borel space equipped with $\mcB(E)$ and $(\Omega,\mcF,\mbP)$ be an atomless probability space supporting $U\sim\mathrm{Unif}[0,1]$. 
    Let $\psi\in\mcU(\hat{E})$.
    Then the following assertions hold:
    \begin{enumerate}[label = (\roman*)]
        \item $\psi$ is measurable.
        \item For every $\alpha \in [0,1]$, 
        \[
            \VaR_\alpha[\psi(x,U)] = \psi(x,\alpha),\quad \forall x\in E.
        \]  
        \item If $Y$ is a random variable on $(\Omega,\mcF,\mbP)$ taking values in $(E,\mcB(E))$, independent of $U$, with law $\nu := \mbP\circ Y^{-1}$, then for every $\alpha \in [0,1]$,
        \begin{align*}
            \VaR_{\alpha}[\psi(Y,U)] = \sup_{\zeta \in \mcZ_{\alpha}(\nu)} \inf_{y\in \mcY} \psi(y,\zeta(y)),
        \end{align*}
        where $\mcZ_{\alpha}(\nu)$ and $\mcY$ are defined as in Theorem \ref{thm:VaRDecomposition}.
    \end{enumerate}
\end{lemma}
\begin{proof}
    (i) To show that $\psi$ is measurable, we must show the sub-level set $\{(x,\alpha)\in E\times[0,1]: \psi(x,\alpha) < c\}\in \mcB(E)\otimes \mcB([0,1])$, for all $c \in \mbQ$. Fix a $c\in\mbQ$ and let $\{\alpha_n\}_{n\in\mbN}$ be the sequence of rational numbers in $[0,1]$. Then it suffices to show that
    \begin{equation}\label{eq-QuantileTransform-1}
         \{(x,\alpha): \psi(x,\alpha) < c\} = \bigcup_{n} \{x: \psi(x,\alpha_n) < c\} \times [0,\alpha_n],
    \end{equation}
    since the RHS is in $\mcB(E)\otimes\mcB([0,1])$, by which the measurability follows. 

    It is clear that the RHS is a subset of the LHS in \eqref{eq-QuantileTransform-1}. To see the opposite inclusion, notice that for any $(x_0,\alpha_0) \in \{(x,a) : \psi(x,\alpha) < c\}$, there exists some rational $\alpha_n > \alpha_0$ such that $\psi(x_0,\alpha_n) < c$. 
    Indeed, by (A2), 
    if $\psi(x_0,\alpha_0) < c$, then there exists some constant $\delta > 0$ such that for all $\alpha \in [\alpha_0,\alpha_0 + \delta) \cap [0,1]$, we have $\psi(x_0,\alpha) < c$. Since $\mbQ\cap[0,1]$ is dense in $[0,1]$, there must exist some $\alpha_n \in \mbQ\cap(\alpha_0,\alpha_0+\delta)\cap[0,1]$ such that $\psi(x_0,\alpha_n) < c$.
    Furthermore,
    \[ (x_0,\alpha_0) \in \{ x: \psi(x,\alpha_n) < c\} \times [0,\alpha_n] \subset \bigcup_n \{ x: \psi(x,\alpha_n) < c\} \times [0,\alpha_n]. \]
    Since $(x_0,\alpha_0) \in \{(x,\alpha): \psi(x,\alpha) < c\}$ is arbitrary, \eqref{eq-QuantileTransform-1} holds.
    
    (ii) The result then follows from the quantile representation. See Lemma B.2 in \cite{Hau2025aistats}.

    (iii) From (i), $\psi$ is measurable. For each $y\in E$, define
    \begin{equation}\label{eq-kernelK}
        K(C|y) := \int_0^1 \bm{1}_{\{\psi(y,u)\in C\}} du,\quad \forall C \in \mcB(\bar{\mbR}).
    \end{equation}
    By the joint measurability, $y\mapsto K(C|y)$ is measurable for each $C\in\mcB(\bar{\mbR})$, and clearly $C\mapsto K(C|y)$ is a probability measure for each $y\in E$. Hence $K:\mcB(\bar{\mbR})\times E \to [0,1]$ is a stochastic kernel.
    Since $Y$ and $U$ are independent, for any $A \in \mcB(E)$ and $C \in \mcB(\bar{\mbR})$, we have
    \begin{align*}
        \mbP[Y \in A, \psi(Y,U) \in C] = \int_A \int_0^1 \bm{1}_{\{\psi(y,u) \in C\}} du \nu(dy) = \int_A K(C|y) \nu(dy). 
    \end{align*}
    This implies that $K(\cdot|y)$ is a regular conditional distribution of $\psi(Y,U)$ given $Y=y$, and this version is defined through \eqref{eq-kernelK} explicitly for each $y\in E$. Hence following \eqref{eq-ConditionalVaR}, we have
    \[
    \VaR_\alpha[\psi(Y,U) \mid Y =y] = \VaR_\alpha[\psi(y,U)] = \psi(y,\alpha),\quad \forall \alpha \in [0,1], y \in E,
    \]
    where the second equality follows from (ii). Since $\psi(Y,U) \in\mbR$ almost surely, then by Theorem \ref{thm:VaRDecomposition}, we obtain the desired result.
    This completes the proof.
\end{proof}

\section{VaR Dynamic Programming in Finite-horizon}\label{sec:VaRDPFinite}
In this section, we present the dynamic programming principle for the finite-horizon VaR MDP. We begin by stating the optimality result associated with the dynamic programming equation, followed by a discussion on the existence of an optimal policy.

Before moving on, we present the notations that will be useful for our theorems. 
Based on the VaR decomposition \eqref{eq-VaRDecomposition}, it is natural to augment the state space $\bfX$ with the risk level $\alpha$.
To this end, let $\hat{\bfX}:=\bfX\times[0,1]$ be the augmented state space of $\bfX$ with $[0,1]$ and $(x,\alpha) \in \hat{\bfX}$ be a state in this augmented space. We also equip $\hat{\bfX}$ with its Borel $\sigma$-algebra $\mcB(\hat{\bfX}) = \mcB(\bfX)\otimes \mcB([0,1])$. 
Define 
\[
\hat{\bfK}:=\{(x,\alpha,a): (x,\alpha) \in \hat{\bfX}, a\in \bfA(x)\},
\] 
being the set of all augmented state-action pairs. Let $\hat{\mbF}$ be the set of all measurable mappings $\hat{f}:\hat{\bfX} \to \bfA$ with $\hat{f}(x,\alpha)\in \bfA(x)$ for all $(x,\alpha) \in \hat{\bfX}$. 

We derive the dynamic programming formulation for the finite-horizon VaR MDP using backward induction. Fix $t \in [T-1]$. We first introduce the notation for the cumulative discounted reward from time $t$ onward. 
Let $\Pi_{t:T-1}$ and $\bfH_{t:T-1}$ denote the sets of all history-dependent policies and histories over the periods $t,\cdots,T-1$, respectively. Specifically, for any $\pi_{t:T-1} := (\varpi_{t:t},\cdots,\varpi_{t:T-1}) \in \Pi_{t:T-1}$ and $k =t,\cdots,T-1$, the stochastic kernel $\varpi_{t:k}(\cdot|h_{t:k})$ specifies a probability distribution over $\bfA(x_k)$ given the history $h_{t:k} := (x_t,a_t,\cdots,x_k) \in \bfH_{t:k}$.
Fix $t\in[T-1]$, and for $\pi_{t:T-1} \in \Pi_{t:T-1}$ and $(x,\alpha) \in \hat{\bfX}$, define 
\begin{equation}\label{eq-hatvpi}
    \hat{v}_{t,\pi_{t:T-1}} (x,\alpha) := \VaR_{\alpha}^{\bar{\mbP}_x^{\pi_{t:T-1},T-t}} \left[ \sum_{j=0}^{T-t-1} \gamma^{j} r_{j+t}(X_{j},A_{j}) + \gamma^{T-t}g(X_{T-t},U_{T-t}) \right],
\end{equation}
where $X_0 = x$, and $U_{T-t}\sim\mathrm{Unif}[0,1]$ is independent of the state-action process. 
Thus, $\hat{v}_{t,\pi_{t:T-1}}(x,\alpha)$ is the $\alpha$-VaR of the cumulative discounted \textit{reward-to-go} over periods $t$ to $T-1$, starting from state $x$ and following policy $\pi_{t:T-1}$. When $t=T$, we denote $\pi_{T:T-1} := \varnothing \in \Pi_{T:T-1}:= \Pi_{0}$, and define
\begin{equation}\label{eq-hatvTpiTT-1}
    \hat{v}_{T,\pi_{T:T-1}}(x,\alpha) := \VaR_\alpha^{\bar{\mbP}_x^{\pi_{T:T-1},0}}\left[ g(X_0,U_0) \right].
\end{equation}
where $U_0 \sim \mathrm{Unif}[0,1]$. For $t \in [T]$, define
\begin{equation}\label{eq-hatv}
    \hat{v}_t (x,\alpha) := \sup_{\pi_{t:T-1} \in \Pi_{t:T-1}} \hat{v}_{t,\pi_{t:T-1}} (x,\alpha),\quad \forall (x,\alpha) \in \hat{\bfX}.
\end{equation}
Then for any fixed $\alpha \in [0,1)$, the finite-horizon VaR MDP objective $J_{\alpha,T}^*(x) = \hat{v}_0(x,\alpha)$, for all $x \in \bfX$.

\subsection{VaR Dynamic Programming Formulation}
We now introduce the dynamic programming equation for the finite-horizon VaR MDP.

For any $v:\hat{\bfX}\to\bar{\mbR}$, $t \in [T-1]$ and a discount factor $\gamma \in (0,1]$, define the operator
\begin{equation}\label{eq-OperatorLt}
\begin{aligned}
    (L_tv)(x,\alpha,a) := & r_t(x,a) + \gamma \sup_{\zeta\in\mcZ_\alpha(Q(\cdot|x,a))} \inf_{x'\in\bfX} v\left(x',\zeta(x')\right),\\
    &\forall (x,\alpha)\in \bfX\times [0,1), a\in \bfA(x), t\in [T-1],
\end{aligned}
\end{equation}
where $\mcZ_\alpha(Q(\cdot|x,a))$ is defined as
\begin{equation}\label{eq-mcZalphaQ}
    \mcZ_\alpha(Q(\cdot|x,a)) := \left\{ \zeta \in \mcL(\bfX): 0 \le \zeta \le 1, \int_{\bfX} \zeta(y) Q(dy|x,a) \le \alpha\right\}, 
\end{equation}
and $(L_tv)(x,1,a) := \infty$, for all $(x,a) \in \bfK$.
Define its maximum reward operator as
\begin{equation}\label{eq-OperatormsLt}
    (\msL_t v) (x,\alpha) := \sup_{a\in \bfA(x)} (L_tv)(x,\alpha,a),\quad \forall (x,\alpha) \in \hat{\bfX}.
\end{equation}
Then we can write a recursive equation in terms of $v$ using $\msL_t$: for all $(x,\alpha) \in \bfX\times[0,1)$,
\begin{equation}\label{eq-Bellman-FH}
    \begin{aligned}
        v_T(x,\alpha) &:= g(x,\alpha),\\
        v_t(x,\alpha) &:= (\msL_t v_{t+1})(x,\alpha), \quad \forall t \in [T-1],
    \end{aligned}
\end{equation}
and $v_t(x,1):=\infty$ for any $x\in\bfX$, $t\in [T]$.
We will see shortly that this is actually the \textit{Bellman optimality equation} for the finite-horizon VaR MDP \eqref{eq-objectivefunction-FH}.

\subsection{Nested VaR Dynamic Programming Formulation}
Lemma \ref{lemma:QuantileTransform} suggests that we can write the VaR dynamic programming \eqref{eq-Bellman-FH} in the form of a uniform random variable.
To see this, for any $v: \hat{\bfX} \to \bar{\mbR}$, $t\in [T-1]$, $\gamma \in (0,1]$, define the operator:
\begin{equation}\label{eq-OperatorBt}
\begin{aligned}
    (B_t v)(x,\alpha,a) :=  r_t(x,a) + \gamma \VaR_{\alpha}^{X'\sim Q(\cdot|x,a)}[v(X',U)],\quad \forall (x,\alpha) \in \hat{\bfX}, t\in [T-1],
\end{aligned}
\end{equation}
where $U\sim \mathrm{Unif}[0,1]$, independent of $X'$, and we use the superscript $X'\sim Q(\cdot|x,a)$ over $\VaR_\alpha$ to specify the distribution of $X'$. The equivalency of $L_t$ and $B_t$ will be shown in Lemma \ref{lemma:msLtmsBtProperties}. Define its maximum reward operator as
\begin{equation}\label{eq-OperatormsBt}
    (\msB_t v)(x,\alpha) := \sup_{a\in\bfA(x)} (B_t v)(x,\alpha,a),\quad \forall (x,\alpha) \in \hat{\bfX}.
\end{equation}
Then define recursively: for all $(x,\alpha) \in \bfX\times[0,1)$,
\begin{equation}\label{eq-BellmanNested-FH}
\begin{aligned}
    v^u_T(x,\alpha) &:= g(x,\alpha),\\
    v^u_t(x,\alpha) &:= (\msB_t v_{t+1}^u)(x,\alpha), \quad \forall t\in [T-1],
\end{aligned}
\end{equation}
and $v^u_t(x,1):=\infty, \forall x\in \bfX$, $t\in [T]$. 
As shown in \cite{Hau2025aistats}, this gives us a form of nested VaR, where the risk levels for subsequent $v_k^u$ at $t < k \le T-1$ are independently and uniformly drawn from $[0,1]$. 

In the following subsection, we establish that, under appropriate assumptions, \eqref{eq-BellmanNested-FH} is equivalent to \eqref{eq-Bellman-FH}, and that both characterize the value function defined in \eqref{eq-hatvpi}. Consequently, the optimal value of the finite-horizon VaR MDP in \eqref{eq-OptimalObjective-FH} can be computed using either dynamic programming formulation \eqref{eq-Bellman-FH} or \eqref{eq-BellmanNested-FH}.

\subsection{Optimality Result}

We impose the following compactness and continuity assumptions on the MDP. These are common assumptions for analysis of MDPs in general spaces. See Section 2.4 in \cite{Bauerle2011book}.

\begin{assump}\label{assump:CompactContinuity}
    \begin{enumerate}[label=(\roman*)]
        \item For each $x\in \bfX$, $\bfA(x)$ is compact, and the multi-function $x \mapsto \bfA(x)$ is upper semicontinuous (u.s.c.). Equivalently, whenever $x_n \to x$ and $a_n \in \bfA(x_n)$, $n\in \mbN$, the sequence $\{a_n\}_{n\in \mbN}$ has a convergent subsequence with limit in $\bfA(x)$.
        
        \item For each $t \in [T-1]$, $r_t: \bfK \to \mbR$ is upper semicontinuous (u.s.c.). That is, for all sequence $\{(x_n,a_n)\}_{n\in\mbN}\subset \bfK$ with $\lim_{n\to\infty}(x_n,a_n) = (x,a) \in \bfK$, it holds
        \[
        \limsup_{n\to\infty} r_t(x_n,a_n) \le r_t(x,a), \quad \forall t\in [T-1].
        \]

        \item The stochastic kernel $Q$ is weakly continuous on $\bfK$. That is, the mapping $(x,a) \mapsto \int_{\bfX} v(y) Q(dy|x,a)$ is continuous on $\bfK$ for every $v \in \mcC_b(\bfX)$.
    \end{enumerate}
\end{assump}

Let $\mcU_u(\hat{\bfX})$ denote the collection of all functions $v\in\mcU(\hat{\bfX})$ such that $v$ is u.s.c. on $\hat{\bfX}$.

We first show the following properties of the operators $L_t$ and $\msL_t$ regarding functions in $\mcU_u(\hat{\bfX})$.

\begin{lemma}\label{lemma:msLtmsBtProperties}
    Under Assumption \ref{assump:CompactContinuity}, for every $t \in [T-1]$, the following results hold:
    \begin{enumerate}[label = (\roman*)]
        \item $L_t v = B_t v$ and hence $\msL_t v = \msB_t v$ for any $v \in \mcU_u(\hat{\bfX})$.
        \item The mappings $(x,\alpha,a) \mapsto (L_t v)(x,\alpha,a)$ and $(x,\alpha,a) \mapsto (B_t v)(x,\alpha,a)$ are u.s.c. on $\hat{\bfK}$, for any $v \in \mcU_u(\hat{\bfX})$.
        \item $\msL_t, \msB_t$ map $\mcU_u(\hat{\bfX})$ to $\mcU_u(\hat{\bfX})$.
        \item For any fixed $v \in \mcU_u(\hat{\bfX})$, $\gamma \in (0,1]$, there exists a measurable selector $\hat{f}_t \in \hat{\mbF}$ that attains the supremum in \eqref{eq-OperatormsLt}. Specifically, for any $(x,\alpha)\in\hat{\bfX}$, it holds that
        \begin{equation}\label{eq-OperatormsLttgMeasurableSelector}
            (\msL_t v)(x,\alpha) = r_t(x,\hat{f}_t(x,\alpha)) + \gamma \sup_{\zeta \in \mcZ_\alpha(Q(\cdot|x,\hat{f}_t(x,\alpha)))} \inf_{x'\in \bfX} v(x', \zeta(x')).
        \end{equation}
    \end{enumerate}
\end{lemma}
\begin{proof}
    (i) When $\alpha = 1$, we have $(L_t v)(x,1,a) = +\infty$ and $(B_t v)(x,1,a) = +\infty$ for any $v\in \mcU_u(\hat{\bfX})$ and $(x,a)\in\bfK$. So the result holds. Now suppose $(x,\alpha) \in \bfX\times[0,1)$.
   By Lemma \ref{lemma:QuantileTransform}, for $v\in\mcU_u(\hat{\bfX})$, we have
    \begin{align*}
        (L_t v)(x,\alpha,a)
        &=r_t(x,a) + \gamma \sup_{\zeta\in \mcZ_\alpha(Q(\cdot|x,a))} \inf_{x'\in\bfX} v(x',\zeta(x'))\\
        &=r_t(x,a) + \gamma \VaR_\alpha^{X'\sim Q(\cdot|x,a)}[v(x',U)] = (B_t v)(x,\alpha,a),\quad \forall t \in[T-1],
    \end{align*}
    where $U\sim \mathrm{Unif}[0,1]$, independent of $X'$. Taking the supremum over $a\in \bfA(x)$ yields the desired result.

    (ii) For any $v \in \mcU_u(\hat{\bfX})$, define
    \begin{equation}\label{eq-Fg}
        F_v(x,\alpha,a) := \VaR_\alpha^{X'\sim Q(\cdot|x,a)}[v(X',U)].
    \end{equation}
    Following (i) and Assumption \ref{assump:CompactContinuity}, it suffices to show that $F_v$ is u.s.c. on $\hat{\bfK}$.
    Hence, for $\{(x_n,\alpha_n,a_n)\}_{n\in\mbN}\subset\hat{\bfK}$ such that $(x_n,\alpha_n,a_n) \to (x,\alpha,a) \in \hat{\bfK}$, we must show that $\limsup_{n\to\infty} F_v(x_n,\alpha_n,a_n)\le F_v(x,\alpha,a)$. When $\alpha = 1$, the claim holds trivially. In the following, we set $\alpha \in [0,1)$.
    
    Let $\lambda$ denote the Lebesgue measure on $[0,1]$. By Assumption \ref{assump:CompactContinuity}(iii) and Theorem 2.8(ii) in \cite{Billingsley1999book}, the product kernel $Q\otimes \lambda$ is also a weakly continuous kernel on $\hat{\bfX}$ given $\hat{\bfK}$.
    Denote
    \[
        \mu_n := Q(\cdot|x_n,a_n) \otimes \lambda,\quad \mu := Q(\cdot|x,a) \otimes \lambda.
    \]
    Then $\mu_n$ converges weakly to $\mu$. Suppose $F_v$ is not u.s.c., and consider $F_v(x,\alpha,a) \in \mbR$. Then there exists a constant $\epsilon > 0$ and a subsequence denoted as $\{(x_n,\alpha_n,a_n)\}_{n\in\mbN}$ without loss of generality that $F_v(x_n,\alpha_n,a_n) > F_v(x,\alpha,a) + \epsilon$. Let $c:= F_v(x,\alpha,a) + \epsilon$. By the definition of VaR and \eqref{eq-Fg}, we have
    \begin{equation}\label{eq-muset}
        \mu(\{(y,u): v(y,u) < c\}) > \alpha.
    \end{equation}
    Meanwhile, from the hypothesis, for sufficiently large $n$, we have $F_v(x_n,\alpha_n,a_n)  > c$, which yields
    \begin{equation}\label{eq-munset}
        \mu_n(\{(y,u): v(y,u) < c\}) \le \alpha_n.
    \end{equation}
    Since $v$ is u.s.c., the sub-level set $G_c:=\{(y,u) : v(y,u) < c\}$ is open. By Portmanteau theorem (Theorem A.43 in \cite{Follmer2016book}), \eqref{eq-muset} and \eqref{eq-munset}, we conclude that
    \[
        \liminf_{n\to\infty}\alpha_n \ge \liminf_{n\to\infty} \mu_n(G_c) \ge \mu(G_c) > \alpha.
    \]
    Since $\alpha_n \to \alpha$, this gives a contradiction. If $F_v(x,\alpha,a) = -\infty$, then take arbitrary $c \in \mbR$, we have $\mu(G_c) > \alpha$, and following a similar argument as above and Portmanteau theorem, for sufficiently large $n$, we have $\mu_n(G_c) > \alpha_n$ and equivalently, $F_v(x_n,\alpha_n,a_n) < c$. Since $c \in \mbR$ is arbitrary, we conclude that $\limsup_{n\to\infty} F_v(x_n,\alpha_n,a_n) = -\infty$. Therefore, $F_v$ is u.s.c. on $\hat{\bfK}$.

    (iii) For any $v \in \mcU_u(\hat{\bfX})$, define $h_t(x,\alpha,a) := (B_t v)(x,\alpha,a)$ and $H_t(x,\alpha) := \sup_{a\in\bfA(x)} h_t(x,\alpha,a)$. Take any sequence $\{(x_n,\alpha_n)\}_{n\in\mbN}$ such that $(x_n,\alpha_n) \to (x,\alpha) \in \hat{\bfX}$. Let $L := \limsup_{n\to\infty} H_t(x_n,\alpha_n)$. If $L = -\infty$, then clearly $L \le H_t(x,\alpha)$. By definition, when $\alpha = 1$, we also have $L \le H_t(x,1) = +\infty$. It remains to consider the case $\alpha \in [0,1)$ and $L \in \mbR\cup\{+\infty\}$. Take any subsequence $\{n_k\}_{k\in\mbN}$ such that $H_t(x_{n_k},\alpha_{n_k}) \to L$. For each $k$, since $\bfA(x_{n_k})$ is compact and by (ii), $a\mapsto h_t(x_{n_k},\alpha_{n_k},a)$ is u.s.c., we can take $a_{n_k} \in \bfA(x_{n_k})$ such that $H_t(x_{n_k},\alpha_{n_k}) = h_t(x_{n_k},\alpha_{n_k},a_{n_k})$. Furthermore, by Assumption \ref{assump:CompactContinuity}, there exists a subsequence, which we still denoted as $\{n_k\}_{k\in\mbN}$ w.l.o.g., and $a\in\bfA(x)$, such that $a_{n_k} \to a$ as $k\to\infty$. Therefore, we have $(x_{n_k},\alpha_{n_k},a_{n_k}) \to (x,\alpha,a)$, and by (ii),
    \begin{align*}
        L = \lim_{k\to\infty} H_t(x_{n_k},\alpha_{n_k}) &= \lim_{k\to\infty} h_t(x_{n_k},\alpha_{n_k},a_{n_k}) 
        \le \limsup_{k\to\infty} h_t(x_{n_k},\alpha_{n_k},a_{n_k}) \\
        &\le h_t(x,\alpha,a) \le \sup_{\bar{a}\in\bfA(x)} h_t(x,\alpha,\bar{a}) = H_t(x,\alpha).
    \end{align*}
    Since $h_t(x,\alpha,a) < + \infty$ for $\alpha \in [0,1)$, this also rules out $L = +\infty$.
    Therefore, $H_t$ is u.s.c.
    Meanwhile, by the definition of VaR, $\alpha\mapsto (\msB_t v)(x,\alpha)$ is non-decreasing; by Lemma \ref{lemma:UscNondecreasingisRightcontinuous}, $\alpha\mapsto H_t(x,\alpha)$ is right-continuous; $(\msB_t v)(x,0)\in [-\infty,\infty)$; $(\msB_t v)(x,1) = +\infty$ for any $x\in\bfX$ and $(\msB_t v)(x,\alpha) \in \mbR$ for $(x,\alpha)\in\bfX\times(0,1)$. Therefore, $\msB_t v \in \mcU_u(\hat{\bfX})$.
    The rest of result then follows from (i).

    (iv) Fix $(x,\alpha) \in \bfX\times[0,1)$ and $t \in [T-1]$. 
    By definition and Assumption \ref{assump:CompactContinuity}, $\bfA(x)$ is nonempty and compact, by (ii) $(x,\alpha,a)\mapsto (L_t v)(x,\alpha,a)$ is u.s.c. for any $v\in\mcU_u(\hat{\bfX})$. Hence $L_t v$ is jointly measurable. By definition, $\hat{\bfK}$ is also a Borel set. Hence by the measurable selection theorem (Theorem A.2.4 in \cite{Bauerle2011book}), 
    there exists a measurable selector $\hat{f}_t: \hat{\bfX} \to \bfA$ such that for any $(x,\alpha) \in \bfX\times[0,1)$, we have $\hat{f}_t(x,\alpha) \in \bfA(x)$, and
    \begin{align*}
        (\msL_t v)(x,\alpha) = r_t(x,\hat{f}_t(x,\alpha)) + \gamma \sup_{\zeta\in\mcZ_\alpha(Q(\cdot|x,\hat{f}_t(x,\alpha)))} \inf_{x'\in \bfX} v(x',\zeta(x')).
    \end{align*}
    For $\alpha = 1$, we define $\hat{f}_t(x,1)$ as an arbitrary fixed selector $f\in \mbF$. Consequently, $\hat{f}_t \in \hat{\mbF}$.
    This completes the proof.
\end{proof}

Next we show several properties of the value functions $v_t$ and $v_t^u$ defined through \eqref{eq-Bellman-FH} and \eqref{eq-BellmanNested-FH} through the following lemma.

\begin{lemma}\label{lemma:vtvtuProperties}
    Suppose Assumption \ref{assump:CompactContinuity} hold and $g \in \mcU_u(\hat{\bfX})$. For each $t\in[T-1]$, let $v_t$ and $v_t^u$ be defined by \eqref{eq-Bellman-FH} and \eqref{eq-BellmanNested-FH}, respectively. Then we have
    \begin{enumerate}[label = (\roman*)]
        \item $v_t, v_t^u \in \mcU_u(\hat{\bfX})$. 
        \item $v_t = v_t^u$.
    \end{enumerate}
\end{lemma}
\begin{proof}
    (i) Clearly $v_T = g \in \mcU_u(\hat{\bfX})$. By Lemma~\ref{lemma:msLtmsBtProperties} (iii), $v_{T-1} = \msL_{T-1} v_T \in \mcU_u(\hat{\bfX})$. Iterating this argument yields $v_t \in \mcU_u(\hat{\bfX})$ for all $t\in [T-1]$. Since $v_T^u = v_T$, Lemma~\ref{lemma:msLtmsBtProperties}(i)(iii) further implies $v_t^u \in \mcU_u(\hat{\bfX})$.
    
    (ii) For $t = T$, by definition we have $v_T = v_T^u$.
    We prove by induction. Suppose $v_{t+1} = v_{t+1}^u$ for some $t \in [T-1]$. Then for $(x,\alpha) \in \bfX\times[0,1)$, we have
    \begin{align*}
        v_t(x,\alpha) = (\msL_t v_{t+1})(x,\alpha) = (\msB_t v_{t+1})(x,\alpha) = (\msB_t v_{t+1}^u)(x,\alpha) = v_t^u(x,\alpha),
    \end{align*}
    where the second equality is from Lemma~\ref{lemma:msLtmsBtProperties} (i) and the third equality is from the induction hypothesis.
    For $\alpha = 1$, we have $v_t(x,\alpha) = v_t^u(x,\alpha) = +\infty$, for all $x\in\bfX$. Hence we conclude that $v_t = v_t^u$ for all $t\in [T]$. This completes the proof.
\end{proof}

We are now ready to present our main theorem.
We show that solving \eqref{eq-Bellman-FH} indeed yields a solution to the VaR objective problem \eqref{eq-OptimalObjective-FH}. This is achieved by establishing the equivalence between \eqref{eq-hatvpi} and \eqref{eq-Bellman-FH}.

\begin{thm}\label{thm:VaRDPTheorem}
    Suppose Assumption \ref{assump:CompactContinuity} hold, $g\in\mcU_u(\hat{\bfX})$, and fix a finite horizon $T \in \mbN_+$. For $t\in [T-1]$, let $\hat{v}_t$, $v_t$ and $v_t^u$ be defined by \eqref{eq-hatv}, \eqref{eq-Bellman-FH} and \eqref{eq-BellmanNested-FH}, respectively. Then, we have
    \[
    \hat{v}_t (x,\alpha) = v_t(x,\alpha) = v_t^u(x,\alpha), \quad \forall (x,\alpha) \in \hat{\bfX}, t\in [T].
    \]
    Consequently, for any fixed $\alpha \in [0,1)$,
    \[
    v_0(x,\alpha) = v_0^u(x,\alpha) = J_{\alpha,T}^*(x),\quad \forall x\in \bfX.
    \]
\end{thm}
\begin{proof}
    The assertion trivially holds for $\alpha = 1$, since all the relevant value functions are defined to be $+\infty$. 

    We prove the result by backward induction. For every $t \in [T]$, we establish both: for all $(x,\alpha) \in \hat{\bfX}$,
    \begin{equation}\label{eq-hatvtpilevt}
        \hat{v}_{t,\pi_{t:T-1}}(x,\alpha) \le v_t(x,\alpha),\quad \forall \pi_{t:T-1} \in \Pi_{t:T-1},
    \end{equation}
    and the existence of an admissible policy $\pi_{t:T-1}^\alpha \in \Pi_{t:T-1}$ satisfying
    \begin{equation}\label{eq-hatvtpialphagevt}
        \hat{v}_{t,\pi_{t:T-1}^\alpha}(x,\alpha) \ge v_t(x,\alpha).
    \end{equation}
    Together, these imply $\hat{v}_t(x,\alpha) = v_t(x,\alpha)$, for all $(x,\alpha) \in \hat{\bfX}$. The result then follows from Lemma \ref{lemma:vtvtuProperties} and \eqref{eq-hatv} that $v_0(x,\alpha) = v_0^u(x,\alpha) = \hat{v}_0(x,\alpha) = J_{\alpha,T}^*(x)$ for all $(x,\alpha) \in\bfX\times[0,1)$. 

    Throughout this proof, conditional laws of future trajectories are taken to be the versions generated by the transition kernel and the corresponding continuation policies through \eqref{eq-IonescuTulceaTheorem}. These are Borel probability kernels and Theorem \ref{thm:VaRDecomposition} therefore applies to these versions.

    \textbf{Step 1:} We begin by introducing auxiliary sets, functions, and policies that will be used in the proof.
    
    Fix $(x,\alpha) \in \bfX\times[0,1)$. By Lemmas \ref{lemma:msLtmsBtProperties}(iv) and \ref{lemma:vtvtuProperties}, for each $t\in [T-1]$, there exists a selector $\hat{f}_t\in\hat{\mbF}$ such that $\hat{f}_t(x,\alpha) \in \bfA(x)$ and
    \begin{equation}\label{eq-vt=rt+gammasupinfvt+1}
        v_t(x,\alpha) = r_t(x,\hat{f}_t(x,\alpha)) + \gamma \sup_{\zeta\in\mcZ_\alpha(Q(\cdot|x,\hat{f}_t(x,\alpha)))} \inf_{x'\in\bfX} v_{t+1}(x',\zeta(x')).
    \end{equation}
    By Theorem \ref{thm:VaRDecomposition}, the supremum in \eqref{eq-vt=rt+gammasupinfvt+1} is attained by some 
    \begin{equation*}
        \bar{\zeta}_{t,x,\alpha}\in \mcZ_\alpha(Q(\cdot|x,\hat{f}_t(x,\alpha))),
    \end{equation*}
    Consequently, for $T\ge 2$, $t\in[T-2]$, and every $x'\in\bfX$,
    \begin{align*}
        \gamma v_{t+1}(x',\bar{\zeta}_{t,x,\alpha}(x')) \ge v_t(x,\alpha) - r_t(x,\hat{f}_t(x,\alpha)),
    \end{align*}
    and the set
    \[
    D_{t+1}(x,\alpha,x') := \left\{ \bar{\alpha} \in [0,1]: \gamma v_{t+1}(x',\bar{\alpha}) \ge v_{t}(x,\alpha) - r_t(x,\hat{f}_t(x,\alpha))\right\},
    \]
    is nonempty. Define
    \begin{equation*}
        \mfa_{t+1}(x,\alpha,x') := \min D_{t+1}(x,\alpha,x'),
    \end{equation*}
    and set $\mfa_{t+1}(x,1,x') := 1$.
    Since $\bar{\zeta}_{t,x,\alpha}(x') \in D_{t+1}(x,\alpha,x')$, we have $\mfa_{t+1}(x,\alpha,x') \le \bar{\zeta}_{t,x,\alpha}(x')$, for all $x'\in\bfX$. It follows that $\int_{\bfX} \mfa_{t+1}(x,\alpha,x') Q(dx'|x,\hat{f}_t(x,\alpha)) \le \alpha$, and hence 
    \begin{equation}\label{eq-mfat+1inmcZalpha}
        \mfa_{t+1}(x,\alpha,\cdot) \in \mcZ_\alpha(Q(\cdot|x,\hat{f}_t(x,\alpha))).
    \end{equation}

    For $s \in [T-1]$, define $\rho_{s:s}^\alpha(h_{s:s}) := \alpha$ and for $s \in [T-2]$, define recursively
    \begin{equation}\label{eq-rhostalpha}
    \begin{aligned}
        \rho_{s:t+1}^\alpha(h_{s:t+1}) := \mfa_{t+1}(x_t,\rho_{s:t}^\alpha(h_{s:t}),x_{t+1}),\quad t=s,\cdots,T-2,
    \end{aligned}
    \end{equation}
    for $h_{s:t+1} = (x_0,a_0,\cdots,x_t,a_t,x_{t+1}) \in \bfH_{s:t+1}$.
    Then following Lemma \ref{lemma:NewPolicyMeasurable}, the stochastic kernels
    \begin{equation}\label{eq-varpistalpha}
        \varpi_{s:t}^\alpha(\cdot|h_{s:t}) := \delta_{\hat{f}_t(x_t,\rho_{s:t}^\alpha(h_{s:t}))}(\cdot),\quad t=s,\cdots,T-1,
    \end{equation}
    define an admissible history-dependent policy
    \begin{equation}\label{eq-pisT-1alpha}
        \pi_{s:T-1}^\alpha := (\varpi_{s:s}^\alpha,\cdots,\varpi_{s:T-1}^\alpha) \in \Pi_{s:T-1}.
    \end{equation}
    When $s = T$, we define $\pi_{T:T-1}^\alpha := \varnothing \in \Pi_{T:T-1}$.

    \textbf{Step 2:} We prove \eqref{eq-hatvtpilevt} by backward induction.

    For $t=T$, let $\pi_{T:T-1}\in\Pi_{T:T-1}$. By \eqref{eq-hatvTpiTT-1} and Lemma \ref{lemma:QuantileTransform}, we have
    \begin{equation}\label{eq-InductionhatvTT-1}
    \begin{aligned}
        \hat{v}_{T,\pi_{T:T-1}}(x,\alpha) = \VaR_\alpha^{\bar{\mbP}_x^{\pi_{T:T-1},0}}[g(X_0,U_0)] = g(x,\alpha) = v_T(x,\alpha),\quad \forall (x,\alpha) \in \hat{\bfX}.
    \end{aligned}
    \end{equation}
    
    Now suppose for every $(x,\alpha) \in \hat{\bfX}$, $t\in[T-1]$, we have
    \begin{equation}\label{eq-hatvt+1pilevt+1}
        \hat{v}_{t+1,\pi_{t+1:T-1}}(x,\alpha) \le v_{t+1}(x,\alpha),\quad \forall \pi_{t+1:T-1} \in \Pi_{t+1:T-1}.
    \end{equation}
    Fix an arbitrary policy $\pi_{t:T-1} \in \Pi_{t:T-1}$ and $(x,\alpha)\in\bfX\times[0,1)$. Conditioning on the initial action, and applying Lemma \ref{lemma:ConditionalVaRInequaltiy} gives
    \begin{equation}\label{eq-hatvtpitT-1}
    \begin{aligned}
        & \hat{v}_{t,\pi_{t:T-1}}(x,\alpha) \\
        &= \VaR_\alpha^{\bar{\mbP}_x^{\pi_{t:T-1},T-t}}\left[r_{t}(x,A_0) + \gamma \sum_{j=1}^{T-t-1}\gamma^{j} r_{j+t}(X_j,A_j) + \gamma^{T-t}g(X_{T-t},U_{T-t}) \right] \\
        &\le \sup_{a\in\bfA(x)}\left\{ r_t(x,a) + \gamma \VaR_\alpha^{\bar{\mbP}_x^{\pi_{t:T-1},T-t}}\left[\sum_{j=1}^{T-t-1}\gamma^{j-1} r_{j+t}(X_j,A_j) + \gamma^{T-t-1}g(X_{T-t},U_{T-t})  \mid A_0 = a\right]  \right\}.
    \end{aligned}
    \end{equation}
    For any fixed $a\in\bfA(x)$, let $\pi_{t+1:T-1}^{(x,a)}$ denote the continuation policy obtained by concatenating the fixed initial state-action pair $(x,a)$ with subsequent histories. Namely, for $\pi_{t:T-1} = (\varpi_{t:t},\varpi_{t:t+1},\cdots,\varpi_{t:T-1})$, we define $\pi_{t+1:T-1}^{(x,a)} := (\varpi_{t+1:t+1}^{(x,a)},\cdots,\varpi_{t+1:T-1}^{(x,a)})$, where for any continuation history 
    \[
    h_{t+1:k} := (x_{t+1},a_{t+1},\cdots,x_k) \in \bfH_{t+1:k},\quad t+1 \le k\le T-1,
    \]
    we have
    \begin{equation}\label{eq-varpit+1(x,a)=varpit(x,a,h)}
        \varpi_{t+1:k}^{(x,a)}(\cdot|h_{t+1:k}) := \varpi_{t:k}(\cdot|(x,a,h_{t+1:k})),\quad t+1 \le k \le T-1.
    \end{equation}
    Since for any $C \in \mcB(\bfA)$, $h_{t+1:k}\mapsto \varpi_{t:k}(C|(x,a,h_{t+1:k}))$ is measurable, we also have $h_{t+1:k}\mapsto \varpi_{t:k}^{(x,a)}(C|h_{t+1:k})$ being measurable.
    Hence $\pi_{t+1:T-1}^{(x,a)}$ is an admissible policy in $\Pi_{t+1:T-1}$.

    Therefore, we have
    \begin{align*}
        &\VaR_\alpha^{\bar{\mbP}_x^{\pi_{t:T-1},T-t}}\left[\sum_{j=1}^{T-t-1}\gamma^{j-1} r_{j+t}(X_j,A_j) + \gamma^{T-t-1} g(X_{T-t},U_{T-t}) \mid A_0 = a\right] \\
        =& \VaR_\alpha^{\bar{\mbP}_x^{\pi_{t:T-1},T-t}}\left[\sum_{j=1}^{T-t-1}\gamma^{j-1} r_{j+t}(X_j,A_j) + \gamma^{T-t-1} g(X_{T-t},U_{T-t}) \mid X_0 = x, A_0 = a\right]\\
        =& \sup_{\zeta\in\mcZ_\alpha(Q(\cdot|x,a))} \inf_{x'\in\bfX} \VaR_{\zeta(x')}^{\bar{\mbP}_x^{\pi_{t:T-1},T-t}}\Bigg[\sum_{j=1}^{T-t-1}\gamma^{j-1} r_{j+t}(X_j,A_j) \\
        &\qquad\qquad\qquad\qquad + \gamma^{T-t-1} g(X_{T-t},U_{T-t}) \mid X_0 = x, A_0 = a, X_1 = x'\Bigg]\\
        =& \sup_{\zeta\in\mcZ_\alpha(Q(\cdot|x,a))} \inf_{x'\in\bfX} \VaR_{\zeta(x')}^{\bar{\mbP}_{x'}^{\pi_{t+1:T-1}^{(x,a)},T-t-1}}\left[\sum_{k=0}^{T-t-2}\gamma^{k} r_{k+t+1}(X_k,A_k) + \gamma^{T-t-1} g(X_{T-t-1},U_{T-t-1}) \right]\\
        =& \sup_{\zeta\in\mcZ_\alpha(Q(\cdot|x,a))} \inf_{x'\in\bfX}\hat{v}_{t+1,\pi_{t+1:T-1}^{(x,a)}}(x',\zeta(x')),
    \end{align*}
    where the second equality follows from Theorem \ref{thm:VaRDecomposition}, and the third equality follows from \eqref{eq-IonescuTulceaTheorem}. Taking back into \eqref{eq-hatvtpitT-1}, and applying the induction hypothesis, we obtain
    \begin{align*}
        \hat{v}_{t,\pi_{t:T-1}}(x,\alpha) &\le \sup_{a\in\bfA(x)}\left\{ r_t(x,a) + \gamma \sup_{\zeta\in\mcZ_\alpha(Q(\cdot|x,a))} \inf_{x'\in\bfX} \hat{v}_{t+1,\pi_{t+1:T-1}^{(x,a)}}(x',\zeta(x'))\right\}\\
        &\le \sup_{a\in\bfA(x)}\left\{ r_t(x,a) + \gamma \sup_{\zeta\in\mcZ_\alpha(Q(\cdot|x,a))} \inf_{x'\in\bfX} v_{t+1}(x',\zeta(x'))\right\}\\
        &= (\msL_t v_{t+1})(x,\alpha) = v_t(x,\alpha).
    \end{align*}
    The case when $\alpha = 1$ trivially holds by definition. Hence, we conclude that $\hat{v}_{t,\pi_{t:T-1}}(x,\alpha) \le v_t(x,\alpha)$ for all $(x,\alpha) \in \hat{\bfX}$ and $\pi_{t:T-1}\in\Pi_{t:T-1}$. This implies  \eqref{eq-hatvtpilevt} holds for all $t\in[T]$.

    \textbf{Step 3:} We prove \eqref{eq-hatvtpialphagevt} by backward induction.

    For $t\in[T]$, consider the policy $\pi_{t:T-1}^\alpha \in \Pi_{t:T-1}$ defined in \eqref{eq-pisT-1alpha}. When $t=T$, similar to \eqref{eq-InductionhatvTT-1}, we have
    \[
    \hat{v}_{T,\pi_{T:T-1}^\alpha}(x,\alpha) = \VaR_\alpha^{\bar{\mbP}_x^{\pi_{T:T-1}^\alpha,0}}[g(X_0,U_0)] = g(x,\alpha) = v_{T}(x,\alpha),\quad \forall(x,\alpha)\in \hat{\bfX}.
    \]
    For $t=T-1$, fix $(x,\alpha) \in \bfX\times[0,1)$, and let $a := \hat{f}_{T-1}(x,\alpha)$. Following Theorem \ref{thm:VaRDecomposition}, Lemma \ref{lemma:msLtmsBtProperties} and \eqref{eq-vt=rt+gammasupinfvt+1}, we have
    \begin{align*}
        \hat{v}_{T-1,\pi_{T-1:T-1}^\alpha}(x,\alpha) &= r_{T-1}(x,a) + \gamma \VaR_\alpha^{\bar{\mbP}_x^{\pi_{T-1:T-1}^\alpha},1}[g(X',U)] \\
        &= r_{T-1}(x,a) + \gamma \sup_{\zeta\in\mcZ_{\alpha}(Q(\cdot|x,a))}\inf_{x'\in\bfX} g(x',\zeta(x')) = v_{T-1}(x,\alpha).
    \end{align*}
    Moreover $\hat{v}_{T-1,\pi_{T-1:T-1}^\alpha}(x,1) = v_{T-1}(x,1)$ for all $x\in\bfX$ by definition.
    Now suppose for every $(x,\alpha)\in\hat{\bfX}$, $T \ge 2$, $t\in[T-2]$, we have
    \begin{equation}\label{eq-hatvt+1pit+1alphagevt}
        \hat{v}_{t+1,\pi_{t+1:T-1}^\alpha}(x,\alpha) \ge v_{t+1}(x,\alpha).
    \end{equation}

    Fix $(x,\alpha) \in \bfX\times[0,1)$, $t\in [T-2]$. Let $a := \hat{f}_t(x,\alpha)$ and define $\hat{\zeta}(x') := \mfa_{t+1}(x,\alpha,x')$ for $x'\in\bfX$. By \eqref{eq-mfat+1inmcZalpha}, $\hat{\zeta} \in \mcZ_\alpha(Q(\cdot|x,a))$. In the following, we show that, after the history $h_{t:t+1} = (x,a,x')$, the continuation of $\pi_{t:T-1}^\alpha$ coincides with the policy $\pi_{t+1:T-1}^{\hat{\zeta}(x')}$.

    Indeed, by \eqref{eq-rhostalpha} and \eqref{eq-varpistalpha}, we have
    \[
    \rho_{t:t+1}^\alpha(h_{t:t+1}) = \mfa_{t+1}(x,\alpha,x') = \hat{\zeta}(x') = \rho_{t+1:t+1}^{\hat{\zeta}(x')}(h_{t+1:t+1}).
    \]
    Now suppose for some fixed $k \in \{t+1,\cdots,T-2\}$, $x'\in\bfX$, and $h_{t+1:k} = (x',a_{t+1},\cdots,x_k)$ being any continuation history, we have
    \begin{equation}\label{eq-rhotk+1alpha=rhot+1k+1hatzetax'}
        \rho_{t:k}^\alpha((x,a,h_{t+1:k})) = \rho_{t+1:k}^{\hat{\zeta}(x')}(h_{t+1:k}).
    \end{equation}
    Then 
    \begin{align*}
        \rho_{t:k+1}^\alpha((x,a,h_{t+1:k+1})) &= \mfa_{k+1}\left(x_k,\rho_{t:k}^\alpha((x,a,h_{t+1:k})),x_{k+1}\right) \\
        &= \mfa_{k+1}\left(x_k, \rho_{t+1:k}^{\hat{\zeta}(x')}(h_{t+1:k}),x_{k+1}\right) = \rho_{t+1:k+1}^{\hat{\zeta}(x')}(h_{t+1:k+1}).
    \end{align*}
    Therefore, we conclude that \eqref{eq-rhotk+1alpha=rhot+1k+1hatzetax'} holds for any $k \in \{t+1,\cdots,T-1\}$. Therefore, by \eqref{eq-varpistalpha}, we have
    \[
    \varpi_{t:k}^\alpha(\cdot|(x,a,h_{t+1:k})) = \delta_{\hat{f}_k(x_k,\rho_{t:k}^\alpha((x,a,h_{t+1:k})))}(\cdot) = \delta_{\hat{f}_k(x_k,\rho_{t+1:k}^{\hat{\zeta}(x')}(h_{t+1:k}))}(\cdot) = \varpi_{t+1:k}^{\hat{\zeta}(x')}(\cdot|h_{t+1:k}).
    \]
    This implies that $\pi_{t+1:T-1}^{\hat{\zeta}(x')}$ is indeed the continuation policy for $\pi_{t:T-1}^\alpha$ after $h_{t:t+1}$.
    
    We have
    \begin{equation}\label{eq-hatvpitT-1alpha}
        \hat{v}_{t,\pi_{t:T-1}^\alpha}(x,\alpha) 
        = r_t(x,a) + \gamma \VaR_\alpha^{\bar{\mbP}_x^{\pi_{t:T-1}^\alpha},T-t}\left[ \sum_{j=1}^{T-t-1} \gamma^{j-1}r_{j+t}(X_j,A_j) + \gamma^{T-t-1} g(X_{T-t},U_{T-t}) \right].
    \end{equation}
    Moreover, by Theorem \ref{thm:VaRDecomposition} and the Markov property, we have
    \begin{align*}
        &\VaR_\alpha^{\bar{\mbP}_x^{\pi_{t:T-1}^\alpha},T-t}\left[ \sum_{j=1}^{T-t-1} \gamma^{j-1}r_{j+t}(X_j,A_j) + \gamma^{T-t-1} g(X_{T-t},U_{T-t}) \right]\\
        =& \sup_{\zeta\in\mcZ_\alpha(Q(\cdot|x,a))}\inf_{x'\in\bfX} \VaR_{\zeta(x')}^{\bar{\mbP}_x^{\pi_{t:T-1}^\alpha},T-t}\Bigg[ \sum_{j=1}^{T-t-1} \gamma^{j-1} r_{j+t}(X_j,A_j) + \gamma^{T-t-1} g(X_{T-t},U_{T-t}) \mid X_1 = x' \Bigg] \\
        =& \sup_{\zeta\in\mcZ_\alpha(Q(\cdot|x,a))}\inf_{x'\in\bfX} \VaR_{\zeta(x')}^{\bar{\mbP}_{x'}^{\pi_{t+1:T-1}^{\hat{\zeta}(x')}},T-t-1}\left[ \sum_{k=0}^{T-t-2} \gamma^k r_{k+t+1}(X_k,A_k) + \gamma^{T-t-1} g(X_{T-t-1},U_{T-t-1}) \right]\\
        =& \sup_{\zeta\in\mcZ_\alpha(Q(\cdot|x,a))}\inf_{x'\in\bfX} \hat{v}_{t+1,\pi_{t+1:T-1}^{\hat{\zeta}(x')}}(x',\zeta(x')).
    \end{align*}
    Combining with \eqref{eq-hatvpitT-1alpha} and the induction hypothesis yields
    \begin{align*}
         \hat{v}_{t,\pi_{t:T-1}^\alpha}(x,\alpha) & = r_t(x,a) + \gamma \sup_{\zeta\in\mcZ_\alpha(Q(\cdot|x,a))}\inf_{x'\in\bfX} \hat{v}_{t+1,\pi_{t+1:T-1}^{\hat{\zeta}(x')}}(x',\zeta(x'))\\
         & \ge r_t(x,a) + \gamma \inf_{x'\in\bfX} \hat{v}_{t+1,\pi_{t+1:T-1}^{\hat{\zeta}(x')}}(x',\hat{\zeta}(x'))\\
         & \ge r_t(x,a) + \gamma \inf_{x'\in\bfX} v_{t+1}(x',\hat{\zeta}(x')).
    \end{align*}
    By the definition of $\hat{\zeta}(x') = \mfa_{t+1}(x,\alpha,x')$, we also have $\gamma v_{t+1}(x',\hat{\zeta}(x')) \ge v_t(x,\alpha) - r_t(x,a)$. This gives
    \[
    \hat{v}_{t,\pi_{t:T-1}^\alpha}(x,\alpha) \ge v_t(x,\alpha).
    \]
    The case for $\alpha = 1$ trivially holds by definition. Hence, we conclude that $\hat{v}_{t,\pi_{t:T-1}^\alpha}(x,\alpha) \ge v_t(x,\alpha)$ for all $t\in [T-1]$. Therefore, \eqref{eq-hatvtpialphagevt} holds. 
    This completes the proof.
\end{proof}

Theorem \ref{thm:VaRDPTheorem} applies to standard Borel state and action spaces, including continuous spaces. Any real-valued u.s.c. terminal reward function $\tilde{g}\in\mcL(\bfX)$ can be embedded in $\mcU_u(\hat{\bfX})$ by setting $g(x,\alpha):=\tilde{g}(x)$ for $\alpha<1$ and $g(x,1):=+\infty$. Neither the one-stage rewards nor the terminal reward need be bounded, allowing the framework to handle unbounded states in a wide range of real-life applications, such as the revenues arising from the price process in the optimal execution problem considered in Section 6.

By \eqref{eq-hatvtpilevt} and \eqref{eq-hatvtpialphagevt} in the proof of Theorem \ref{thm:VaRDPTheorem}, it suffices to consider the policy defined in \eqref{eq-pisT-1alpha} to characterize an optimal policy. In particular, \eqref{eq-mfat+1inmcZalpha} shows that, under the optimal policy, a sequence of state-action pairs uniquely determines the associated trajectory of ``optimal" risk levels as in \eqref{eq-rhostalpha}. Conversely, this risk-level trajectory uniquely determines the optimal policy on the history space via \eqref{eq-varpistalpha}. We formalize it as the following corollary.

\begin{cor}\label{cor:MarkovHistoryPolicy}
    Suppose Assumption \ref{assump:CompactContinuity} holds. Let $\hat{f}_t \in \hat{\mbF}$ be a selector satisfying
    \begin{equation}\label{eq-MarkovHistoryPolicy-1}
        \hat{f}_t(x,\alpha) \in \argmax_{a\in\bfA(x)} (L_t v_{t+1})(x,\alpha,a),\quad \forall (x,\alpha) \in \hat{\bfX}.
    \end{equation}
    For a fixed initial risk level $\alpha \in [0,1)$, set $\rho_0^\alpha(h_0) := \alpha$ for any $h_0 = (x_0) \in \bfH_0$. Given any history $h_{t+1} = (x_0,a_0,\cdots,x_t,a_t,x_{t+1})\in\bfH_{t+1}$, recursively define
    \begin{equation}\label{eq-MarkovHistoryPolicy-2}
        \rho_{t+1}^\alpha(h_{t+1}) := \min\Bigg\{ \bar{\alpha} \in [0,1]:  \gamma v_{t+1}(x_{t+1},\bar{\alpha}) \ge  v_t(x_t,\rho_t^\alpha(h_t)) - r_t(x_t,\hat{f}_t(x_t,\rho_t^\alpha(h_t)) \Bigg\}, 
    \end{equation}
    for $t = 0,\cdots,T-2$. Let $\pi_T^\alpha := (\varpi_0^\alpha,\cdots,\varpi_{T-1}^\alpha)$ induced by
    \begin{equation}\label{eq-varpitalphaht}
        \varpi_t^\alpha(\cdot|h_t) := \delta_{\hat{f}_{t}(x_t,\rho_t^\alpha(h_t))}(\cdot).
    \end{equation}
    Then $\pi_T^\alpha \in \Pi_T$, and $\hat{v}_{0,\pi_T^\alpha}(x,\alpha) = v_0(x,\alpha) = J_{\alpha,T}^*(x)$. Thus, $\pi_T^\alpha$ is optimal for the finite horizon problem \eqref{eq-objectivefunction-FH}. The same conclusion holds when $v_t$ and $L_t$ are replaced by $v_t^u$ and $B_t$, respectively.
\end{cor}
\begin{proof}
    From Lemma \ref{lemma:NewPolicyMeasurable}, the mappings $\rho_t^\alpha$ are measurable, and the kernels in \eqref{eq-varpitalphaht} define an admissible policy $\pi_T^\alpha \in \Pi_T$. Step 2 and 3 in the proof of Theorem \ref{thm:VaRDPTheorem} show that $\hat{v}_{0,\pi_T^\alpha}(x,\alpha) \le v_0(x,\alpha)$ and $\hat{v}_{0,\pi_T^\alpha}(x,\alpha) \ge v_0(x,\alpha)$ for all $(x,\alpha)\in\hat{\bfX}$. Therefore, by Theorem \ref{thm:VaRDPTheorem}, we have $\hat{v}_{0,\pi_T^\alpha}(x,\alpha) = v_0(x,\alpha) = J_{\alpha,T}^*(x)$, i.e., the induced policy $\pi_T^\alpha$ is optimal. The result then follows from Lemmas \ref{lemma:msLtmsBtProperties} and \ref{lemma:vtvtuProperties}.
\end{proof}

Notice that, by construction, the selector $\hat f_t$ is Markovian with respect to the augmented state space $\hat{\bfX}$. By \eqref{eq-MarkovHistoryPolicy-2}, the optimal risk level $\rho_{t+1}^\alpha(h_{t+1})$ depends on the history $h_{t+1}$ only through $\rho_t^\alpha(h_t)$ and the state transition $(x_t,x_{t+1})$. Hence, when expressed on the original state space $\bfX$, the optimal policy $\pi_T^\alpha$ is generally history-dependent. While for the CVaR discounted MDP in \cite{Bauerle2011mmor}, the augmented state represents the accumulated discounted reward, in the VaR formulation, the optimal risk level in \eqref{eq-MarkovHistoryPolicy-2} generally does not admit a closed-form representation, highlighting an additional complexity relative to the CVaR case.

\begin{remark}\label{remark:DegenerateCaseofalpha=0}
    When $\alpha = 0$, VaR reduces to the essential infimum. Hence, by Lemma \ref{lemma:QuantileTransform}, \eqref{eq-Bellman-FH} and \eqref{eq-BellmanNested-FH}
     reduce to
    \[
        v_t(x,0) = \max_{a\in\bfA(x)}\left\{ r_t(x,a) + \gamma \essinf_{X'\sim Q(\cdot|x,a)} v_{t+1}(X',0) \right\},\quad \forall x\in\bfX, t\in [T-1],
    \]
    and \eqref{eq-MarkovHistoryPolicy-2} is automatically satisfied by $\bar{\alpha} = 0$ since by the upper semicontinuity of $v_{t+1}(\cdot,0)$, we have
    \[
    \gamma v_{t+1}(x',0) \ge v_t(x_t,0) - r_t(x_t,\hat{f}_t(x_t,0)),
    \]
    for any $x' \in \supp(Q(\cdot|x_t,\hat{f}_{t}(x_t,0)))$. 
    Thus, defining $\tilde{v}_t(x) := v_t(x,0)$ yields a recursion independent of the risk level, with $\tilde{v}_0$ solving the optimal essential-infimum problem.     
    Consequently, the optimal policy defined by selectors $\tilde{f}_t(x) := \hat{f}_t(x,0)$ is optimal and deterministic Markovian on the original state space $\bfX$.
\end{remark}

\section{VaR Dynamic Programming in Infinite-horizon}\label{sec:VaRDPInfinite}

In this section, we investigate the infinite-horizon discounted formulation \eqref{eq-ObjectiveFunction-IH} of the VaR MDP. Specifically, we consider a discount factor $\gamma \in (0,1)$ and assume that the reward function $r$ is time-invariant and finite. 

\begin{assump}\label{assump:StationaryBoundedReward}
    There exists a bounded u.s.c. function $r: \bfK \to \mbR$ such that $r_t \equiv r$ for all $t\ge 0$. In particular, there exists $\bar{r} < \infty$ such that $|r(x,a)| \le \bar{r}$, for all $(x,a) \in \bfK$.
\end{assump}
Under Assumption \ref{assump:StationaryBoundedReward}, we have $J_{\alpha,\infty}^*(x) \le \frac{\bar{r}}{1-\gamma} < \infty$ for $\alpha \in [0,1)$ and hence the infinite-horizon discounted objective \eqref{eq-ObjectiveFunction-IH} is well-defined.

We now introduce a dynamic programming recursion analogous to \eqref{eq-Bellman-FH} initialized from time $0$ onward. When the reward function is time-invariant, we write $L, \msL, B, \msB$ for the operators $L_t, \msL_t, B_t, \msB_t$ defined in \eqref{eq-OperatorLt}, \eqref{eq-OperatormsLt}, \eqref{eq-OperatorBt} and \eqref{eq-OperatormsBt}, respectively. The recursion is then defined as follows: for all $(x,\alpha)\in\bfX\times[0,1)$,
\begin{equation}\label{eq-wt}
\begin{aligned}
    w_{0}(x,\alpha) &:= 0,\\
    w_{t}(x,\alpha) &:= (\msL w_{t-1})(x,\alpha), \quad \forall t \in \mbN_+, 
\end{aligned}
\end{equation}
with $w_t(x,1) := \infty$ for all $x\in\bfX$ and $t\in\mbN$.
Similarly, the nested VaR dynamic programming recursion initialized from time $0$ onward can be defined as: for all $(x,\alpha) \in \bfX\times[0,1)$,
\begin{equation}\label{eq-wut}
\begin{aligned}
    w_{0}^u(x,\alpha) &:= 0,\\
    w_{t}^u(x,\alpha) &:= (\msB w_{t-1}^u)(x,\alpha),\quad \forall t\in \mbN_+, 
\end{aligned}
\end{equation}
with $w_t^u(x,1) := \infty$, for all $x\in \bfX$ and $t\in\mbN$.

To analyze the properties of \eqref{eq-wt} and \eqref{eq-wut}, we consider the following function set. Let $\mcU_b(\hat{\bfX})$ denote the collection of all functions $g:\hat{\bfX}\to\bar{\mbR}$ such that
\begin{enumerate}[label = (B\arabic*)]
    \item $g$ is u.s.c. on $\hat{\bfX}$;
    \item $\alpha \mapsto g(x,\alpha)$ is non-decreasing and right-continuous for each $x\in \bfX$;
    \item There exists a constant $M >0$ such that $|g(x,\alpha)| \le M$ for all $(x,\alpha) \in \bfX\times[0,1)$;
    \item $g(x,1) = +\infty$, for all $x\in\bfX$.
\end{enumerate}
We define the following metric for functions in $\mcU_b(\hat{\bfX})$:
\[
    d(u,v) := \sup_{(x,\alpha) \in \bfX\times[0,1)} | u(x,\alpha) - v(x,\alpha)|,\quad \forall u,v \in \mcU_b(\hat{\bfX}).
\]
It is straightforward to verify the following properties of the operator $\msB$ over the metric space $(\mcU_b(\hat{\bfX}), d)$.

\begin{lemma}\label{lemma:msBProperties}
    Under Assumptions \ref{assump:CompactContinuity}(i)(iii) and \ref{assump:StationaryBoundedReward}, the following assertions hold:
    \begin{enumerate}[label = (\roman*)]
        \item The metric space $(\mcU_b(\hat{\bfX}), d)$ is complete;
        \item $\msB: \mcU_b(\hat{\bfX}) \to \mcU_b(\hat{\bfX})$;
        \item $d(\msB u, \msB v) \le \gamma d(u, v)$, for $u, v \in \mcU_b(\hat{\bfX})$.
    \end{enumerate}
\end{lemma}
\begin{proof}
    (i) First note that $d$ is indeed a metric on $\mcU_b(\hat{\bfX})$. By (B3), for any $u,v \in \mcU_b(\hat{\bfX})$, $u,v$ are finite and bounded and hence $d(u,v) < \infty$. The symmetry and triangular inequality follows immediately. Moreover, if $d(u,v) = 0$, then $u(x,\alpha) = v(x,\alpha)$ for $(x,\alpha) \in \bfX\times[0,1)$ and by definition $u(x,1) = v(x,1) = +\infty$ for all $x\in \bfX$. Thus $u=v$.
    
    Let $\{v_n\}_{n\in\mbN} \subset \mcU_b(\hat{\bfX})$ be a Cauchy sequence under metric $d$. For any $(x,\alpha) \in \bfX\times[0,1)$ and $\epsilon > 0$, there exists some $N > 0$ such that $|v_n(x,\alpha) - v_m(x,\alpha)|\le d(v_n,v_m) \le \epsilon$ for all $n,m > N$ and hence $\{v_n(x,\alpha)\}_{n\in\mbN}$ is Cauchy in $\mbR$. Define $v(x,\alpha) := \lim_{n\to\infty} v_n(x,\alpha)$ for $(x,\alpha)\in\bfX\times[0,1)$ and set $v(x,1) := +\infty$ for all $x\in\bfX$. Then for any $\epsilon > 0$, there exists some $N > 0$ such that $d(v_n,v_m) \le \epsilon$. Let $m\to\infty$, we obtain $|v_n(x,\alpha) - v(x,\alpha)| \le \epsilon$ for all $(x,\alpha)\in\bfX\times[0,1)$ and hence $d(v_n,v) \le \epsilon$ for all $n > N$. This implies $v_n \to v$ uniformly on $\bfX\times[0,1)$. The case for $\alpha = 1$ trivially holds. Hence it remains to check that $v\in \mcU_b(\hat{\bfX})$.

    Since $v_n \to v$ on $\bfX\times[0,1)$ and $\{v_n\}_{n\in\mbN}$ are u.s.c. It follows that $v$ is also u.s.c. on $\hat{\bfX}$. Indeed, the result trivially holds for $\alpha = 1$. Fix an arbitrary $(x_0,\alpha_0) \in \bfX\times[0,1)$. For any $\epsilon > 0$, choose $n$ so large such that $d(v_n,v) < \frac{\epsilon}{3}$. Since $v_n$ is u.s.c. at $(x_0,\alpha_0)$, we have $\limsup_{(x,\alpha)\to(x_0,\alpha_0)}v_n(x,\alpha) < v_n(x_0,\alpha_0) + \frac{\epsilon}{3}$. Consequently,
    \[
    \limsup_{(x,\alpha)\to(x_0,\alpha_0)} v(x,\alpha) < \limsup_{(x,\alpha)\to(x_0,\alpha_0)} v_n(x,\alpha) + \frac{\epsilon}{3} < v_n(x_0,\alpha_0) + \frac{2\epsilon}{3} < v(x_0,\alpha_0) + \epsilon.
    \]
    Since $\epsilon > 0$ is arbitrary, we conclude that $v$ is u.s.c. on $\bfX\times[0,1)$.

    For any $x\in\bfX$, since $\alpha\mapsto v_n(x,\alpha)$ is non-decreasing, we have $v_n(x,\alpha_1) \le v_n(x,\alpha_2)$ for any $0\le \alpha_1\le \alpha_2 \le 1$. Taking $n\to\infty$, we obtain that $v(x,\alpha_1) \le v(x,\alpha_2)$. Hence $\alpha\mapsto v(x,\alpha)$ is also non-decreasing. The right-continuity then follows from Lemma~\ref{lemma:UscNondecreasingisRightcontinuous}.

    Finally, to show $v$ is also uniformly bounded on $\bfX\times[0,1)$, choose $N$ sufficiently large such that $d(v_n,v_N) < 1$ and by (B3), there exists some constant $M_N > 0$ such that $|v_N(x,\alpha)|\le M_N$ for all $(x,\alpha)\in\bfX\times[0,1)$. We have 
    \[
    |v_n(x,\alpha)| \le |v_n(x,\alpha) - v_N(x,\alpha)| + |v_N(x,\alpha)| \le 1 + M_N < \infty,\quad \forall (x,\alpha)\in\bfX\times[0,1).
    \]
    Let $n\to\infty$, we obtain that $|v(x,\alpha)| \le 1+ M_N$ for all $(x,\alpha)\in\bfX\times[0,1)$.
    
    Therefore we conclude that $(\mcU_b(\hat{\bfX}),d)$ is complete.

    (ii) Take any $v \in \mcU_b(\hat{\bfX})$ and $M \ge \max\{ M_v, \frac{\bar{r}}{1-\gamma}\}$, where $M_v := \sup_{(x,\alpha)\in\bfX\times[0,1)}|v(x,\alpha)|$. Clearly, we have $|(\msB v)(x,\alpha)| \le \bar{r} + \gamma M_v \le M$. Hence (B3) holds for $\msB v$.
    Since $\mcU_b(\hat{\bfX})\subseteq\mcU_u(\hat{\bfX})$, by Lemma \ref{lemma:msLtmsBtProperties}, we have $\msB v \in \mcU_u(\hat{\bfX})$, hence (B1)(B2)(B4) hold for $\msB v$.
    Therefore $\msB v \in \mcU_b(\hat{\bfX})$, and the statement is shown.

    (iii) For any $v_1, v_2 \in \mcU_b(\hat{\bfX})$ and $(x,\alpha) \in \bfX\times[0,1)$, let $\eta := d(v_1,v_2)$, which is finite.
    For any $x\in \bfX$, $a\in \bfA(x)$ and $\alpha \in [0,1)$, let $X'\sim Q(\cdot|x,a)$, $U\sim \mathrm{Unif}[0,1]$, independent of $X'$. Since $U < 1$ almost surely, we have $|v_1(X',U) - v_2(X',U)| \le \eta$ almost surely. By the monotonicity and translation invariance of VaR, we obtain that
    \[
        \left|\VaR_\alpha[v_1(X',U)] - \VaR_\alpha[v_2(X',U)]\right| \le \eta.
    \]
    Following the supremum inequality, for any $(x,\alpha) \in \bfX\times[0,1)$, we have
    \begin{align*}
        |(\msB v_1)(x,\alpha) - (\msB v_2)(x,\alpha)| \le \gamma \sup_{a\in\bfA(x)} |\VaR_\alpha[v_1(X',U)] - \VaR_\alpha[v_2(X',U)]| \le \gamma \eta.
    \end{align*}
    Taking the supremum over $\bfX\times[0,1)$, it follows that $d(\msB v_1, \msB v_2) \le \gamma d(v_1,v_2)$.
\end{proof}

We summarize below the stationary-reward dynamic programming result. This shows that one can iterate the dynamic programming equations \eqref{eq-wt} or \eqref{eq-wut}, and as the iterates goes to infinity, we can recover the infinite horizon objective \eqref{eq-ObjectiveFunction-IH}.

\begin{thm}\label{thm:VaRDPInfiniteHorizon}
    Suppose Assumptions \ref{assump:CompactContinuity}(i)(iii) and \ref{assump:StationaryBoundedReward} hold. Then there exists a unique $w^* \in \mcU_b(\hat{\bfX})$ such that $w^* = \msB w^*$. For $w_t$ and $w_t^u$ defined in \eqref{eq-wt} and \eqref{eq-wut},
    we have
    \[
    w^*(x,\alpha) = \lim_{T\to\infty} w_T(x,\alpha) = \lim_{T\to\infty} w_T^u(x,\alpha) = J_{\alpha,\infty}^*(x),\quad \forall (x,\alpha) \in \bfX\times[0,1),
    \]
    and $w^*(x,1) = +\infty$ for all $x\in\bfX$.
\end{thm}
\begin{proof}
    Following Lemma \ref{lemma:msBProperties}, by Banach's fixed point theorem, there exists a unique fixed point $w^* \in \mcU_b(\hat{\bfX})$ such that $w^* = \msB w^*$ and $w^* = \lim_{T\to\infty} w_T^u = \lim_{T\to\infty} \msB^T w_0^u$.
    
    Since $\mcU_b(\hat{\bfX}) \subset \mcU_u(\hat{\bfX})$, by Lemma~\ref{lemma:msLtmsBtProperties}(i), we have $\msL w = \msB w$ for any $w \in \mcU_b(\hat{\bfX})$. Moreover, $w_0(x,\alpha) = w_0^u(x,\alpha) = 0$ for $(x,\alpha) \in \bfX\times[0,1)$ and $w_0(x,1) = w_0^u(x,1) = +\infty$, we conclude that $w_T = w_T^u$ for all $T \ge 0$, and $\lim_{T\to\infty} w_T = w^*$. Following Theorem \ref{thm:VaRDPTheorem}, we also have $w_T(x,\alpha) = J_{\alpha,T}^*(x)$ with $g\equiv 0$ in \eqref{eq-OptimalObjective-FH} for all $x\in\bfX$.

    Fix $(x,\alpha)\in\bfX\times[0,1)$. For any $\pi\in \Pi$, define the random variables $Z_T^\pi:= \sum_{t=0}^{T-1} \gamma^t r(X_t,A_t)$ and $Z_\infty^\pi := \sum_{t=0}^\infty \gamma^t r(X_t, A_t)$, where the actions follow the policy $\pi$ and $X_0 = x$, $X_{t+1} \sim Q(\cdot|X_t,A_t)$. By Assumption \ref{assump:StationaryBoundedReward}, we have $|Z_\infty^\pi - Z_T^\pi| \le \sum_{t=T}^\infty \gamma^t \bar{r} = \frac{\bar{r}\gamma^T}{1-\gamma} =: \epsilon_T$. Let $\pi_{0:T-1} \in \Pi_T$ be the truncated policy of $\pi$ from time $0$ to $T$, and $J_{\alpha,T}(x,\pi_{0:T-1})$ be the objective function \eqref{eq-objectivefunction-FH} with $g\equiv 0$.
    By the Lipschitz property of VaR, we have
    \begin{align*}
        |J_{\alpha,\infty}(x,\pi) - J_{\alpha,T}(x,\pi_{0:T-1})| = |\VaR_\alpha[Z_\infty^\pi] - \VaR_\alpha[Z_T^\pi]| \le \norm{Z_\infty^\pi - Z_T^\pi}_\infty \le \epsilon_T.
    \end{align*}
    This implies that 
    \begin{align*}
        J_{\alpha,\infty}(x,\pi) \le J_{\alpha,T}(x,\pi_{0:T-1}) + \epsilon_T \le \sup_{\pi_T \in \Pi_T} J_{\alpha,T}(x,\pi_T) + \epsilon_T = w_T(x,\alpha) + \epsilon_T.
    \end{align*}
    Taking the supremum over $\pi \in \Pi$, we obtain that $J_{\alpha,\infty}^*(x) \le w_T(x,\alpha) + \epsilon_T$ for all $(x,\alpha) \in \bfX\times[0,1)$. 
    Conversely, for any $\pi_T \in \Pi_T$, it can always been augmented into an infinite horizon policy $\bar{\pi} \in \Pi$ such that $\bar{\pi}_{0:T-1} = \pi_T$. Clearly, we also have $|J_{\alpha,T}(x,\pi_T) - J_{\alpha,\infty}(x,\bar{\pi})| \le \epsilon_T$. Hence
    \[
    J_{\alpha, T}(x,\pi_T) \le J_{\alpha,\infty}(x,\bar{\pi}) + \epsilon_T \le J_{\alpha,\infty}^*(x) + \epsilon_T.
    \]
    Taking the supremum over $\pi_T \in \Pi_T$, we obtain that $w_T(x,\alpha) \le J_{\alpha,\infty}^*(x) + \epsilon_T$ for all $(x,\alpha) \in \bfX\times[0,1)$. Therefore, we have $|w_T(x,\alpha) - J_{\alpha,\infty}^*(x)| \le \epsilon_T$. 
    Since $\epsilon_T \to 0$ as $T\to\infty$, we conclude that $\lim_{T\to\infty} w_T(x,\alpha) = J_{\alpha,\infty}^*(x)$.
    This completes the proof.
\end{proof} 

We have a similar result regarding the optimal policy to Corollary~\ref{cor:MarkovHistoryPolicy}.
\begin{cor}
    Suppose Assumptions \ref{assump:CompactContinuity}(i)(iii) and \ref{assump:StationaryBoundedReward} hold. Let $w^*$ be the unique fixed point in Theorem \ref{thm:VaRDPInfiniteHorizon}. 
    Then there exists a selector $\hat{f} \in \hat{\mbF}$ satisfying
    \begin{equation}\label{eq-hatfinargmaxLwstar}
        \hat{f}(x,\alpha) \in \argmax_{a\in\bfA(x)} (L w^*)(x,\alpha,a).
    \end{equation}
    For a fixed initial risk level $\alpha \in [0,1)$, set $\rho_0^\alpha(h_0) := \alpha$ for any $h_0 = (x_0) \in \bfH_0$. Given any history $h_{t+1} = (x_0,a_0,\cdots,x_t,a_t,x_{t+1}) \in \bfH_{t+1}$, recursively define
    \begin{equation}\label{eq-rhotalphaforwstar}
        \rho_{t+1}^\alpha (h_{t+1}) := \min\left\{ \bar{\alpha} \in [0,1] : \gamma w^*(x_{t+1},\bar{\alpha}) \ge w^*(x_t,\rho_t^\alpha(h_t)) - r(x_t,\hat{f}(x_t,\rho_t^\alpha(h_t)) \right\},
    \end{equation}
    for $t\in\mbN$. Let $\pi^\alpha := (\varpi_0^\alpha,\varpi_1^\alpha,\cdots)$ induced by
    \[
    \varpi_t^\alpha(\cdot|h_t) := \delta_{\hat{f}(x_t,\rho_t^\alpha(h_t))}(\cdot).
    \]
    Then $\pi^\alpha \in \Pi$, and $J_{\alpha,\infty}(x,\pi^\alpha) = w^*(x,\alpha) = J_{\alpha,\infty}^*(x)$ for all $x\in\bfX$. Thus, $\pi^\alpha$ is optimal for the infinite horizon problem \eqref{eq-ObjectiveFunction-IH}. The same conclusion holds when $L$ is replaced by $B$.
\end{cor}
\begin{proof}
    Since $w^* \in \mcU_b(\hat{\bfX})\subseteq\mcU_u(\hat{\bfX})$, Lemma \ref{lemma:msLtmsBtProperties}(iv) guarantees the existence of a selector $\hat{f}\in\hat{\mbF}$ satisfying \eqref{eq-hatfinargmaxLwstar}. Moreover, by Lemma \ref{lemma:msLtmsBtProperties}(i), we have $Lw^* = Bw^*$, so the same selector is greedy for both formulations. 

    Since $w^* = L w^*$, the construction in Step 1 of the proof of Theorem \ref{thm:VaRDPTheorem} applies with $v_t = v_{t+1} := w^*$, $L_t := L$, $r_t := r$, and $\hat{f}_t := \hat{f}$. Consequently, for every $(x,\alpha,x') \in \bfX\times[0,1)\times\bfX$, the set
    \[
    D(x,\alpha,x') := \left\{ \bar{\alpha} \in [0,1]: \gamma w^*(x',\bar{\alpha}) \ge w^*(x,\alpha) - r(x,\hat{f}(x,\alpha))\right\},
    \]
    is nonempty, and its minimum $\mfa(x,\alpha,x') := \min D(x,\alpha,x')$ is well-defined. Moreover, $\mfa(x,\alpha,\cdot) \in \mcZ_\alpha(Q(\cdot|x,\hat{f}(x,\alpha)))$. By Lemma \ref{lemma:NewPolicyMeasurable}, $\mfa$ is measurable, and the recursively defined mapping $\rho_t^\alpha$ are also measurable, and the kernels $\varpi_t^\alpha(\cdot|h_t)$ define an admissible policy $\pi^\alpha \in \Pi$.

    Fix $T \in \mbN_+$ and let $U_T\sim\mathrm{Unif}[0,1]$ be independent of the controlled state-action process. Define
    \[
    Z_T^{\pi^\alpha} := \sum_{t=0}^{T-1} \gamma^t r(X_t,A_t) + \gamma^T w^*(X_T,U_t),
    \]
    where the action process $\{A_t\}_{t=0}^{T-1}$ follows the policy $\pi^\alpha$ from time $0$ to $T-1$.
    Consider the finite-horizon problem with terminal value $\bar{v}_T(x,\alpha) := w^*(x,\alpha)$ for all $(x,\alpha) \in \hat{\bfX}$, and backward recursion $\bar{v}_t = \msB \bar{v}_{t+1}$ for $t \in [T-1]$. Since $w^* = \msB w^*$, we have $\bar{v}_t = w^*$ for all $t\in [T]$. Therefore, the selector $\hat{f}$ is greedy at every stage, and the risk-level update in Corollary \ref{cor:MarkovHistoryPolicy} is exactly \eqref{eq-rhotalphaforwstar}. Since $w^* \in \mcU_b(\hat{\bfX}) \subseteq \mcU(\hat{\bfX})$, $\bar{v}_t$ is equivalent to $v_t^u$ in \eqref{eq-BellmanNested-FH} with $v_T^u = g = w^*$.
    Then by Theorem \ref{thm:VaRDPTheorem} and Corollary \ref{cor:MarkovHistoryPolicy}, we have
    \begin{equation}\label{eq-tildev=VaRsumgammar+gammaTwstar}
        \bar{v}_0(x,\alpha) = \VaR_\alpha^{\bar{\mbP}_x^{\pi_T^\alpha}}\left[ \sum_{t=0}^{T-1} \gamma^t r(X_t,A_t) + \gamma^T w^*(X_T,U_T) \right] = w^*(x,\alpha),\quad \forall (x,\alpha) \in \hat{\bfX}.
    \end{equation}
    where $\pi_T^\alpha$ is equivalent to $\pi^\alpha$ being restricted from $0$ to $T-1$, i.e., $\pi_T^\alpha = \pi^\alpha_{0:T-1}$.

    Let 
    \[
    Z_\infty^{\pi^\alpha} := \sum_{t=0}^\infty \gamma^t r(X_t,A_t),\quad
    Z_T^{\pi_T^\alpha} := \sum_{t=0}^{T-1} \gamma^t r(X_t,A_t) + \gamma^T w^*(X_T,U_T),
    \]
    where the system processes are driven by policies $\pi^\alpha$ and $\pi_T^\alpha$, respectively.
    Since $|r(x,a)| \le \bar{r}$, $w^* \in \mcU_b(\hat{\bfX})$ and from Theorem \ref{thm:VaRDPInfiniteHorizon}, $|w^*(x,\alpha)| = |J_{\alpha,\infty}^*(x)| \le \frac{\bar{r}}{1-\gamma}$ for $(x,\alpha)\in\bfX\times[0,1)$, we have almost surely,
    \begin{align*}
        |Z_T^{\pi_T^\alpha} - Z_\infty^{\pi^\alpha}| \le \gamma^T |w^*(X_T,U_T)| + \sum_{t=T}^\infty \gamma^t \bar{r} \le \frac{\bar{r} \gamma^T}{1-\gamma} + \frac{\bar{r}\gamma^T}{1-\gamma} = \frac{2\bar{r}\gamma^T}{1-\gamma}.
    \end{align*}
    By the Lipschitz property of VaR, we have
    \[
    |\VaR_\alpha[Z_T^{\pi_T^\alpha}] - J_{\alpha,\infty}(x,\pi^\alpha)| \le \frac{2\bar{r}\gamma^T}{1-\gamma}.
    \]
    Letting $T\to\infty$, combining with \eqref{eq-tildev=VaRsumgammar+gammaTwstar}, we obtain that $w^*(x,\alpha) = J_{\alpha,\infty}(x,\pi^\alpha)$. Thus by Theorem \ref{thm:VaRDPInfiniteHorizon}, $J_{\alpha,\infty}(x,\pi^\alpha) = J_{\alpha,\infty}^*(x)$, so $\pi^\alpha$ is optimal. For $\alpha = 1$, the result follows directly from the convention that both sides equal $+\infty$. This completes the proof.
\end{proof}

\section{Application: VaR Optimal Execution}\label{sec:VaRExecution}
In this section, we consider an optimal execution problem with market impact in a continuous state space under the VaR criterion, building on the framework proposed in \cite{Almgren2001jor}.

Suppose we initially hold $x_0:=\bar{x}\ge 0$ shares of an infinitely divisible asset (e.g., a cryptocurrency) and seek to liquidate the entire position by time $T$. The horizon is divided into $N$ equal stages, with decision stage $k\in\{1,\cdots,N\}$. Let $x_k$ denote the remaining inventory and $n_k\ge 0$ the number of shares sold at stage $k$, subject to the terminal liquidation constraint $x_N=0$. The asset price $S_k$ evolves stochastically and is affected by market impact, so larger sales induce greater price declines. Our objective is to determine a liquidation strategy that maximizes the $\alpha$-VaR of total execution revenue for a given $\alpha\in(0,1)$.

At each stage $k\ge 1$, the asset price follows the arithmetic random walk
\[
    S_k = S_{k-1} + \sigma \sqrt{\tau} Z_k -\tau g\left(\frac{n_k}{\tau}\right),
\]
where $\sigma > 0$ is the volatility parameter, $\{Z_k\}_{k\in\mbN_+}$ are i.i.d. standard normal random variables, and $g(v)$ denotes the permanent market impact at trading rate $v:=n_k/\tau$, with $\tau:=T/N$.
Temporary market impact is incorporated through the effective execution price
\[
    \tilde S_k:=S_{k-1}-h\left(\frac{n_k}{\tau}\right),
\]
where $h(v)$ captures the transient price concession induced by the current trade and does not affect subsequent asset prices.
At stage $k$, we first execute $n_k$ shares at price $\tilde{S}_k$ and then observe the next market price $S_{k+1}$.
Accordingly, the total revenue generated by the liquidation strategy is $\sum_{k=1}^N n_k\tilde S_k$.
Table \ref{table-3StageExecutionModel} illustrates the evolution of asset prices and inventories in a three-stage execution model.

\begin{table}[htbp]
    \centering
    \caption{Price and share evolutions of a three-stage execution model}
    \label{table-3StageExecutionModel}
    \resizebox{\textwidth}{!}{
    \begin{tabular}{ccccc}
    \toprule 
    Stage & 0 & 1 & 2 & 3\\
    \midrule
    Market price & $S_0$  & $S_1 = S_0 + \sigma\sqrt{\tau}Z_1 - \tau g(n_1/\tau)$ & $S_2 = S_1 + \sigma \sqrt{\tau}Z_2 - \tau g(n_2/\tau)$ &  \\
    Execution price & &  $\tilde{S}_1 = S_0 - h(n_1/\tau)$ & $\tilde{S}_2 = S_1 - h(n_2/\tau)$ & $\tilde{S}_3 = S_2 - h(n_3/\tau)$\\
    Inventory & $x_0 = \bar{x}$ & $x_1 = \bar{x} - n_1$ & $x_2 = \bar{x} - n_1 - n_2$ & $x_3 = 0$\\
    \bottomrule
    \end{tabular}
    }
\end{table}

We consider an $N$-stage execution problem with linear market impact: $g(v) = \kappa v$ and $h(v) = \epsilon + \eta v$, for fixed parameters $\kappa$, $\eta$ and $\epsilon$. Note that when $N=1$, the only strategy available is to sell everything at stage 1, i.e., $n_1 = \bar{x}$, which is less interesting. Therefore in the following, we assume that $N \ge 2$.

The corresponding risk-neutral problem reduces to a deterministic quadratic program \citep{Almgren2001jor}. In contrast, the VaR objective yields a substantially more challenging stochastic optimization problem.

In the following, We develop two solution approaches. The first optimizes over static strategies, which is standard in discrete-time optimal execution, but is generally suboptimal. The second approach applies the VaR dynamic programming framework developed in Section~\ref{sec:VaRDPFinite}. We show that the resulting dynamic strategies outperform their static counterparts while remaining computationally tractable.

\subsection{Static Strategy}\label{sec:VaROEStatic}
We first consider static strategies, under which all trading decisions are determined prior to the start of execution. The optimal static strategy can be characterized through the efficient frontier associated with the analytical solution of the mean-variance model (see Section 3.2 of \cite{Almgren2001jor}). In general, however, no closed-form solution is available. In this subsection, we show that an analytical solution for the optimal static strategy can be derived in the special case when $\kappa \ge 2\eta/\tau$. In the next subsection, we further show that the optimal dynamic strategy also admits an analytical solution under this regime, thereby allowing for an explicit comparison between the static and dynamic strategies.

An admissible static strategy is represented by $\bmn := (n_1,\cdots, n_N)^\top \in \mbR^{N}$, where $n_k \ge 0$ for all $1\le k\le N$ and $\sum_{k=1}^N n_k = \bar{x}$. The inventory at stage $k$ is $x_k = \bar{x} - \sum_{i=1}^k n_i$ and $x_N = 0$.
Then following Section 1.5 in \cite{Almgren2001jor}, the total profit under this strategy can be written as
\begin{equation*}
\begin{aligned}
    R(\bmn) &:= \sum_{k=1}^N n_k \tilde{S}_k
    = \left(S_0 - \epsilon - \frac{1}{2}\kappa \bar{x}\right)\bar{x} - \left( \frac{\eta}{\tau} - \frac{\kappa}{2} \right) \sum_{k=1}^{N} n_k^2 + \sigma \sqrt{\tau} \sum_{k=1}^{N-1} x_k Z_k \\
    & \sim \mcN\left(\mfA(\bmn), \sigma^2 \tau \sum_{k=1}^{N-1} x_k^2 \right),
\end{aligned}
\end{equation*}
where $\mfA(\bmn)$ is the deterministic component of $\bmn$, defined as
\begin{equation*}
\begin{aligned}
    \mfA(\bmn) := &  C(S_0,\bar{x}) - \frac{\mfp}{2} \left( \bar{x}^2 - \sum_{k=1}^N n_k^2 \right),
\end{aligned}
\end{equation*}
where $C(s,x) := (s - \epsilon - \frac{\eta x}{\tau})x$ for all $s,x \in \mbR$, and $\mfp:= \kappa- \frac{2\eta}{\tau}$.
The $\alpha$-VaR of $R(\bmn)$ is then
\begin{equation}\label{eq-OEStaticStrategy}
    \VaR_\alpha[R(\bmn)] = \mfA(\bmn) + \sigma \sqrt{\tau \sum_{k=1}^{N-1} x_k^2} \Phi^{-1}(\alpha).
\end{equation}

When $\mfp < 0$, $\mfA$ is concave in $n_k$. The risk term (second term) in \eqref{eq-OEStaticStrategy} is concave for $\alpha<1/2$ and convex for $\alpha>1/2$, precluding a closed-form solution in general. At $\alpha=1/2$, this term vanishes, yielding an optimal equal-trading solution $n_k =\bar{x}/N$ for all $1\le k\le N$. Otherwise, the problem must typically be solved numerically or via mathematical programming software.

When $\mfp \ge 0$, noting that $\sum_{k=1}^N n_k = \bar{x}$, and $\{n_i\}_{i=1}^N$ are non-negative, we have
\begin{align*}
    \mfA(\bmn) = C(S_0,\bar{x}) + \frac{\mfp}{2} \left( \sum_{k=1}^N n_k^2 - \bar{x}^2 \right) \le C(S_0,\bar{x}).
\end{align*}
Moreover, since $x_k \in [0,\bar{x}]$, 
\[
    0\le \sqrt{\tau \sum_{k=1}^{N-1} x_k^2} \le \bar{x}\sqrt{\tau(N-1)},
\]
with the second equality attained if and only if $n_N = \bar{x}$.
For $\alpha < 1/2$, $\VaR_\alpha[R(\bmn)] \le C(S_0,\bar{x})$, with equality attained by taking $n_1 = \bar{x}$, which eliminates the risk term. For $\alpha = 1/2$, the same upper bound holds for $\VaR_\alpha[R(\bmn)]$ and the equality is attained by letting $n_k = \bar{x}$ for one $k \in \{1,\cdots,N\}$.
For $\alpha > 1/2$, we have $\VaR_\alpha[R(\bmn)] \le C(S_0,\bar{x}) + \sigma\bar{x}\sqrt{\tau(N-1)}\Phi^{-1}(\alpha)$, with equality attained by setting $n_N = \bar{x}$.
As a result,
\begin{equation}\label{eq-OEStaticObj}
\begin{aligned}
    J_{\mathrm{stat}}^*(S_0,\bar{x},\alpha) &:= \max_{\bmn} 
    \VaR_\alpha[R(\bmn)] 
    = C(S_0,\bar{x}) + \sigma \bar{x} \sqrt{\tau(N-1)} (\Phi^{-1}(\alpha))_+.
\end{aligned}
\end{equation}
Therefore, when $\mfp \ge 0$, the optimal static strategy is a one-shot liquidation policy: all shares are sold at the initial stage for $\alpha<0.5$ and at the terminal stage for $\alpha>0.5$. At $\alpha=0.5$, the liquidation may occur at any stage.
This also includes the special case $\kappa=\eta=0$, corresponding to zero market impact.

\subsection{Dynamic Strategy}\label{sec:VaROEDynamic}
Now we consider deriving an optimal dynamic strategy by applying the framework in Section \ref{sec:VaRDPFinite}. Let $s, x$ stand for the current market price and inventory level.
Following \eqref{eq-Bellman-FH}, at stage $N$, since the only option is to sell all the inventory, we can define the value function as 
\begin{align*}
    v_N(s,x,\theta) := \left(s - h\left( \frac{x}{\tau} \right)\right) x
    = \left(s - \epsilon  - \frac{\eta x}{\tau} \right)x = C(s,x),
\end{align*}
where $\theta \in [0,1)$ is the risk level. 

At stage $N-1$, suppose the number of shares to be executed is $n \in [0,x]$. Then define
\begin{equation}\label{eq-funqN-1}
\begin{aligned}
    q_{N-1}(s,x,\theta,n) &:= \left(s - h\left(\frac{n}{\tau}\right)\right)n +\VaR_\theta\left[v_N\left(s + \sigma\sqrt{\tau}Z_{N-1} - \tau g\left(\frac{n}{\tau}\right), x - n ,U_N\right)\right]\\
    &= \mfp n^2 -\mfp x n + C(s,x) + \sigma\sqrt{\tau}(x-n)\mfB_{N-1}(\theta),\quad \theta \in [0,1),
\end{aligned}
\end{equation}
where $\{U_k\}_{k=1}^N$ is a sequence of i.i.d. random variables with $U_k \sim \mathrm{Unif}[0,1]$, and $\mfB_{N-1}(\theta) := \VaR_\theta[Z_{N-1}]$ for $\theta \in [0,1)$.

In the following, we consider the two regimes determined by the sign of $\mfp$. Specifically, Regime 1 corresponds to $\mfp \ge 0$, whereas Regime 2 corresponds to $\mfp < 0$.

\subsubsection{Regime 1}\label{sec:VaROEDynamicRegime1}
When $\mfp \ge 0$, $q_{N-1}$ is convex on $n$, and its maximum is attained at the boundaries. For $\theta \in [0,1)$, define
\begin{align*}
    v_{N-1}(s,x,\theta) &:= \max_{n \in [0,x]} q_{N-1}(s,x,\theta,n)
    = C(s,x) + \sigma x \sqrt{\tau} (\mfB_{N-1}(\theta))_+.
\end{align*}
This implies that if $\theta > 0.5$, the optimal strategy at stage $N-1$ is to wait until the next stage, and liquidate all the inventory otherwise. 

A similar representation extends to all stages. Specifically, following \eqref{eq-Bellman-FH}, the Bellman recursion and backward induction yield, for $k = N-2,\cdots, 1$, 
\begin{equation}\label{eq-DynamicRegime1qk}
    \begin{aligned}
        q_k(s,x,\theta,n) &:= \left(s - h\left(\frac{n}{\tau}\right)\right)n + \VaR_\theta\left[v_{k+1}\left(s+\sigma \sqrt{\tau}Z_{k} - \tau g\left(\frac{n}{\tau}\right), x-n, U_{k+1}\right)\right]\\
        &=\mfp n^2 - \mfp xn + C(s,x) + \sigma\sqrt{\tau}(x-n) \mfB_{k}(\theta), \quad \theta \in [0,1), n \in [0,x],
    \end{aligned}
\end{equation}
where
\begin{equation}\label{eq-mfBk}
    \mfB_{k}(\theta) := \VaR_\theta\left[ Z_k + (\mfB_{k+1}(U_{k+1}))_+ \right], \quad \theta \in [0,1),
\end{equation}

For $\theta \in (0,1)$, $q_k$ is convex in $n$ and therefore the value function is defined as
\begin{equation}\label{eq-DynamicRegime1vk}
    v_k(s,x,\theta) := \max_{n \in [0,x]} q_k(s,x,\theta,n)
    = C(s,x) + 
    \sigma x \sqrt{\tau} (\mfB_k(\theta))_+.
\end{equation}
At $\theta = 0$, $\mfB_k(0) = -\infty$, the optimal action is to liquidate all the inventory immediately, which gives 
\[
    v_k(s,x,0) := \max_{n\in[0,x]} q_k(s,x,0,n) = q_k(s,x,0,x) = C(s,x).
\]
This implies that \eqref{eq-DynamicRegime1vk} also holds for all $\theta \in [0,1)$.
An optimal policy therefore waits if $\mfB_k(\theta) > 0$ and liquidates all the inventory otherwise. The execution problem thus reduces to an optimal stopping problem, with objective value
\begin{equation}\label{eq-VaROEDynamicObjective}
    J_{\mathrm{dyna}}^*(S_0,\bar{x},\alpha) := v_1(S_0, \bar{x}, \alpha) = C(S_0,\bar{x})  + 
    \sigma \bar{x} \sqrt{\tau} (\mfB_1(\alpha))_+.
\end{equation}

To implement the optimal policy, we need to obtain optimal risk levels $\{\theta_k\}_{k=1}^{N-1}$ along the history. Set $\theta_1 := \alpha$. 
Suppose the optimal action is to wait at stage $1 \le k\le N-2$, i.e., $n_i^* = 0$ for $i=1,\cdots,k$, and the current market price, inventory and risk level are $(s_k,x_k,\theta_k)$ with a realization of $Z_k$ as $z_k \in \mbR$.
Then following Corollary \ref{cor:MarkovHistoryPolicy}, the optimal risk level for the next stage is given by
\begin{align*}
    \theta_{k+1} := \min\left\{ \bar{\theta} \in [0,1]: v_{k+1}(s_k + \sigma\sqrt{\tau} z_k, x_k, \bar{\theta}) \ge q_k(s_k,x_k,\theta_k,0) \right\}.
\end{align*}
Note that each $\mfB_k$ is continuous and strictly increasing on $(0,1)$ with range $\mbR$, since the underlying distribution has an everywhere positive continuous density. Following \eqref{eq-DynamicRegime1vk} and \eqref{eq-DynamicRegime1qk}, we have
\begin{align*}
    \theta_{k+1} = \begin{cases}
        \mfB_{k+1}^{-1}(\mfB_k(\theta_k) - z_k), & z_k < \mfB_k(\theta_k),\\
        0, & z_k \ge \mfB_k(\theta_k).
    \end{cases}
\end{align*}
In particular, all risk levels reached before liquidation lie in $[0,1)$.

This recursion yields a threshold rule based directly on cumulative price innovations. To formalize this, given a fixed threshold $b\in\mbR$, define the first passing time of the cumulative price innovation that reaches or exceeds $b$ as
\[
    t_b := \inf\left\{ k \in \mbN: S_k \ge b\right\},\quad S_k := \sum_{i=1}^k Z_i, \quad S_0:= 0.
\]
Then a threshold rule can be defined as: liquidate all the inventory at stage $1 + \min\{t_b, N-1\}$, i.e., liquidate immediately after the threshold is reached, or at the terminal stage if no earlier liquidation occurs.
The following result shows that the optimal dynamic strategy can be recovered by choosing $b = \mfB_1(\alpha)$.  

\begin{prop}\label{prop:DynamicThresholdStrategy}
    For every $N \ge 2$ and $\alpha \in (0,1)$, $\max_{b\in\mbR} \VaR_\alpha\left[ S_{t_b \wedge (N-1)} \right] = (\mfB_1(\alpha))_+$,
    with the maximum attained at $b = \mfB_1(\alpha)$.
    Moreover, $\mfB_1(\alpha) = \VaR_\alpha[M_{N-1}]$ for all $N \ge 2$., where $M_{N-1} := \max_{1\le k\le N-1} S_k$.
\end{prop}
\begin{proof}
    For $k \in \mbN$, $b\in\mbR$, and $1\le m\le N-1$, denote
    \begin{equation}\label{eq-defnSkRmYm}
        R_m(b) := S_{t_b\wedge m}, \quad Y_m := Z_{N-m} + (\bar{Y}_{m-1})_+,
    \end{equation}
    with $\bar{Y}_{m-1}$ as an independent copy of $Y_{m-1}$ and $Y_0 := 0$. Set $M_0 := 0$.

    We first show that 
    \begin{equation}\label{eq-mfB1=VaRYN-1}
        \mfB_1(\alpha) = \VaR_\alpha[Y_{N-1}],\quad \forall \alpha \in (0,1).
    \end{equation}
    We prove it using mathematical induction. Note that by definition, we have $\mfB_{N-1}(\alpha) = \VaR_\alpha[Z_{N-1}]$ for $\alpha \in (0,1)$. Now assume it holds for some $1\le m \le N-2$ that
    \begin{equation}\label{eq-mfBNm=VaRuYm}
        \mfB_{N-m}(u) = \VaR_u[Y_m] =: q_m(u),\quad \forall u \in (0,1).
    \end{equation}
    Then by the quantile transformation (Lemma A.23 in \cite{Follmer2016book}), we have $q_m(U_{N-m}) \overset{d}{=} Y_m$, where $\overset{d}{=}$ stands for equality in distribution. Then following \eqref{eq-mfBk}, this implies that
    \begin{align*}
        \VaR_\alpha[Y_{m+1}] &= \VaR_\alpha[Z_{N-m-1} + (\bar{Y}_m)_+] = \VaR_\alpha[Z_{N-m-1} + (q_m(U_{N-m}))_+] \\
        &= \VaR_\alpha[Z_{N-m-1} + (\mfB_{N-m}(U_{N-m}))_+] = \mfB_{N-m-1}(\alpha).
    \end{align*}
    Hence we conclude that \eqref{eq-mfBNm=VaRuYm} holds for $1 \le m \le N-1$ and \eqref{eq-mfB1=VaRYN-1} follows.
    
    We next show that $Y_{N-1} \overset{d}{=}M_{N-1}$. We prove by induction that $Y_m \overset{d}{=} M_m$ for $1 \le m \le N-1$. Indeed, for $m=1$, we have $Y_1 = Z_{N-1} \overset{d}{=} Z_1 = M_1$. For $m\ge 2$, define $\bar{M}_{m-1} := \max_{1 \le k \le m-1} \sum_{i=2}^{k+1}Z_i$. Then $\bar{M}_{m-1} \overset{d}{=} M_{m-1}$, independent of $Z_1$, and satisfies
    \begin{equation}
        M_m = \max\{ Z_1, Z_1 + Z_2, \cdots, Z_1 + \cdots+ Z_m\} = Z_1 + (\bar{M}_{m-1})_+.
    \end{equation}
    If $Y_{m-1} \overset{d}{=} M_{m-1}$, then following \eqref{eq-defnSkRmYm}, the fact that $Z_1\overset{d}{=}Z_{N-m}$, we obtain that $Y_m \overset{d}{=} M_m$. Hence the claim holds for every $1\le m\le N-1$. Combining with \eqref{eq-mfB1=VaRYN-1}, we conclude that $\mfB_1(\alpha) = \VaR_\alpha[M_{N-1}]$.

    For any $b \in \mbR$, we have $R_m(b) \le \max_{0\le j\le m} S_j = (M_m)_+$. By the monotonicity and positive homogeneity of VaR, we have $\VaR_\alpha[R_m(b)] \le \VaR_\alpha[(M_m)_+] = (\mfB_1(\alpha))_+$. It remains to show the attainment. 
    
    Set $b = \mfB_1(\alpha)$. If $b \le 0$, then $t_b = 0$, so $R_m(b) = 0$, attaining the upper bound.

    Now suppose $b > 0$. Define $F_m(b,z) := \mbP[R_m(b) < z]$ and $F_{Y_m}(z) := \mbP[Y_m < z]$ for $z \in \mbR$ and $1\le m\le N-1$. We show by induction that
    \begin{equation}\label{eq-Fmbb=FYmb}
        F_m(b,b) = F_{Y_m}(b),\quad \forall 1\le m\le N-1.
    \end{equation}
    The identity clearly holds for $m = 1$. Let $m \ge 2$, suppose the equality holds for $m-1$.
    Note that $R_m(b) = Z_1 + \bm{1}_{\{Z_1 < b\}} \bar{R}_m(b-Z_1)$ where we define $\bar{R}_m(z) := \bar{S}_{\bar{t}_z \wedge m}$, $\bar{S}_k := \sum_{i=2}^{k+1} Z_i$ for $1\le k\le m$ and $\bar{S}_0 := 0$, $\bar{t}_z := \inf\{ k\ge 0: \bar{S}_k \ge z\}$, for all $z\in\mbR$. Clearly, by construction, we have $\bar{R}_m(b) \overset{d}{=} R_m(b)$ and is independent of $Z_1$.
    Then we have
    \begin{equation}\label{eq-IntegralRecursion}
    \begin{aligned}
        F_m(b,z) &= \int_{-\infty}^\infty \mbP[R_m(b) < z \mid Z_1 = t] \varphi(t) dt
        = \int_{-\infty}^\infty \mbP[\bar{R}_{m-1}(b) < z - Z_1 \mid Z_1 = t] \varphi(t) dt\\
        &= \int_{-\infty}^b F_{m-1}(b-t,z-t)\varphi(t) dt + \bm{1}_{\{z > b\}} \int_b^z \varphi(t) dt,\quad \forall z\in\mbR.
    \end{aligned}
    \end{equation}
    Take $z = b$, we obtain that
    \[
    F_m(b,b) = \int_{-\infty}^b F_{m-1}(b-t,b-t) \varphi(t) dt = \int_{-\infty}^b F_{Y_{m-1}}(b-t) \varphi(t) dt = F_{Y_m}(b),
    \]
    where the second equality follows from the induction hypothesis, and the last inequality follows from the recursion of $Y_m$ (see \eqref{eq-defnSkRmYm}). Hence \eqref{eq-Fmbb=FYmb} holds.

    Moreover, the integral recursion \eqref{eq-IntegralRecursion} and the initial case $F_1(b,z) = \Phi(z)$ imply by induction that $z\mapsto F_m(b,z)$ is continuous and strictly increasing for every $b > 0$. Therefore, taking $m = N-1$ and $b = \mfB_1(\alpha) = \VaR_\alpha[Y_m] > 0$, we obtain that $F_m(b,b) = F_{Y_m}(b) = \alpha$. Therefore, $\VaR_\alpha[R_m(b)] = b = (\mfB_1(\alpha))_+$. Combining with the case of $b \le 0$, this completes the proof.
\end{proof}

\subsubsection{Regime 2}
When $\mfp < 0$, $q_{N-1}$ in \eqref{eq-funqN-1} becomes a concave function over $n$, and the maximal value over $n$ is attained at 
\begin{equation}\label{eq-nN-1star}
    n_{N-1}^*(x,\theta) :=  \min\left\{ x, \max\left\{0, \frac{x}{2} + \frac{\sigma\sqrt{\tau}\mfB_{N-1}(\theta)}{2\mfp}\right\}\right\}.
\end{equation}
Actually, the optimal point is attained at $0$ if $\theta > \Phi(\frac{-\mfp x}{\sigma\sqrt{\tau}})$; at $x$ if $\theta < \Phi(\frac{\mfp x}{\sigma\sqrt{\tau}})$; and at $\frac{x}{2} + \frac{\sigma\sqrt{\tau}\Phi^{-1}(\theta)}{2\mfp}$ if $\theta$ is in between.
We have
\begin{align*}
    v_{N-1}(s,x,\theta) := \max_{n\in[0,x]} q_{N-1}(s,x,\theta,n)
    = C(s,x)  + \mfC_{N-1}(x,\theta),
\end{align*}
where $\mfC_{N-1}(x,\theta)$ is defined as
\begin{equation}\label{eq-mfCN-1}
\begin{aligned}
    \mfC_{N-1}(x,\theta) :=\begin{cases}
        \sigma x \sqrt{\tau} \Phi^{-1}(\theta), & \theta \in  (\Phi(\frac{-\mfp x}{\sigma\sqrt{\tau}}),1), \\
        \frac{1}{4}\left( x\sqrt{-\mfp} + \frac{\sigma\sqrt{\tau}\Phi^{-1}(\theta)}{\sqrt{-\mfp}} \right)^2,  & \theta \in [ \Phi(\frac{\mfp x}{\sigma\sqrt{\tau}}), \Phi(\frac{-\mfp x}{\sigma\sqrt{\tau}})],  \\
        0, & \theta \in [0,\Phi(\frac{\mfp x}{\sigma\sqrt{\tau}})).
    \end{cases} 
\end{aligned}
\end{equation}

Such a decomposition extends to all stages. Specifically, for $k=N-2,\cdots,1$, denote
\begin{equation}\label{eq-RecursionRegime2}
    \begin{aligned}
         \mfD_k(y,\theta) &:= \VaR_\theta[\sigma\sqrt{\tau} y Z_k + \mfC_{k+1}(y,U_{k+1})],\\
        q_k(s,x,\theta,n) &:= C(s,x) - \mfp n(x - n) + \mfD_k(x-n,\theta),\\
        \mfC_k(x,\theta) &:= \max_{n\in[0,x]} \{ -\mfp n(x-n) + \mfD_k(x-n,\theta) \}.
    \end{aligned}
\end{equation}
Then we have for $k=1,\cdots,N-1$,
\begin{equation}\label{eq-RecursionRegime2vk}
    v_k(s,x,\theta) := \max_{n\in[0,x]} q_k(s,x,\theta,n) = C(s,x) + \mfC_k(x,\theta).
\end{equation}

Now we consider recovering the optimal risk-level at stage $k+1$ given state $(s,x,\theta)$ at stage $k \in \{1,\cdots,N-2\}$. Following Corollary \ref{cor:MarkovHistoryPolicy}, we need to find $\theta_{k+1}$ such that
\begin{equation*}
    \theta_{k+1} = \min \left\{ \bar{\theta} \in [0,1] : \begin{gathered}
        \left(s - h\left(\frac{n_k^*}{\tau}\right)\right)n_k^* + v_{k+1}\left(s + \sigma\sqrt{\tau}z_k - \tau g\left(\frac{n_k^*}{\tau}\right), x - n_k^*, \bar{\theta} \right) \\
        \ge v_k(s,x,\theta)
    \end{gathered} \right\},
\end{equation*}
where $n_k^*$ is the optimal stage $k$ decision, i.e., $q_k(s,x,\theta,n_k^*) = v_k(s,x,\theta)$, and $z_k$ is a realization of $Z_k$. Modifying the terms, we obtain that
\begin{equation}\label{eq-RecursionRegime2thetak+1}
    \theta_{k+1} = \min\left\{ \bar{\theta} \in [0,1] 
    : \sigma\sqrt{\tau} z_k (x - n_k^*) +\mfC_{k+1}(x-n_k^*,\bar{\theta}) \ge \mfD_k(x - n_k^*,\theta)
    \right\}
\end{equation}
The value function for stage $k+1$ is then $v_{k+1}(s + \sigma\sqrt{\tau} z_k - \kappa n_k^*, x - n_k^*,\theta_{k+1})$.

\subsection{Comparison of Static and Dynamic Strategies}
In this subsection, we compare the optimal static and dynamic strategies derived in Sections \ref{sec:VaROEStatic} and \ref{sec:VaROEDynamic}. In Regime 1, both strategies admit analytical comparison. In Regime 2, only numerical solutions are available. Our results show that, in both cases, the dynamic strategy outperforms its static counterpart.

\subsubsection{Regime 1}

In this case, the difference between the optimal static and dynamic objectives in \eqref{eq-OEStaticObj} and \eqref{eq-VaROEDynamicObjective} is characterized by $\sigma\bar{x}\sqrt{T}\Delta_N(\alpha)$, where $\Delta_N(\alpha)$ is defined as
\[
    \Delta_N(\alpha) := \frac{J_{\mathrm{dyna}}^*(S_0,\bar{x},\alpha) - J_{\mathrm{stat}}^*(S_0,\bar{x},\alpha)}{\sigma\bar{x}\sqrt{T}} = \frac{(\mfB_1(\alpha))_+ - \sqrt{N-1}(\Phi^{-1}(\alpha))_+}{\sqrt{N}},
\]
where $\mfB_1(\alpha)$ also depends on $N$. We call $\Delta_N(\alpha)$ the normalized $\alpha$-VaR performance improvement.

When $N=2$, $\mfB_1(\alpha)=\Phi^{-1}(\alpha)$ for $\alpha \in (0,1)$, so the optimal static and dynamic strategies are equivalent. Starting from $N \ge 3$, the dynamic strategy differs from the static strategy, since the second-stage decision depends on the realization of $Z_1$. 
In fact, we can show that, when $\alpha$ is greater than some threshold $\alpha_0(N)$, the optimal dynamic strategy strictly outperforms its static counterpart. We call $\alpha_0(N)$ the critical risk level for the dynamic strategy.

\begin{prop}\label{prop:DynamicOutperformStatic}
    For every $N \ge 2$ and $\alpha \in (0,1)$, $\Delta_N(\alpha) \ge 0$. Moreover, $\Delta_2(\alpha) = 0$ while for $N \ge 3$, $\Delta_N(\alpha) > 0 \Leftrightarrow \alpha > \alpha_0(N)$ where $\alpha_0(N) := \frac{1}{4^{N-1}}\binom{2N-2}{N-1}$.
\end{prop}
\begin{proof}
    Since $N \ge 2$, following the definition of $\Delta_N$, we only need to compare $\mfB_1(\alpha)$ and $\sqrt{N-1}(\Phi^{-1}(\alpha))_+$.
    By Proposition \ref{prop:DynamicThresholdStrategy},
    $\mfB_1(\alpha)=\VaR_\alpha[M_{N-1}]$. Since $M_{N-1}\ge S_{N-1}$ almost surely
    and $S_{N-1}\sim\mcN(0,N-1)$, we have $\mfB_1(\alpha)\ge\VaR_\alpha[S_{N-1}]
    =\sqrt{N-1}\Phi^{-1}(\alpha)$ for all $N\ge 2$.

    Now suppose $N\ge3$. The vector $(S_1,S_{N-1})$ is a non-degenerate Gaussian vector with an everywhere positive density. Thus, for every $z\in\mbR$,
    \begin{align*}
        \mbP[S_{N-1} < z] - \mbP[M_{N-1} < z] = \mbP[S_{N-1} < z, M_{N-1} \ge z] \ge \mbP[S_{N-1} < z, S_1 \ge z] > 0.
    \end{align*}
    Since $M_{N-1}$ and $S_{N-1}$ have continuous, strictly increasing cumulative distribution functions, it follows that $\mfB_1(\alpha) > \sqrt{N-1}\Phi^{-1}(\alpha)$ for all $\alpha \in (0,1)$.
    Consequently, $\Delta_N(\alpha) > 0$ for $\alpha > 1/2$.

    For $\alpha \le 1/2$, we have $\Delta_N(\alpha) = (\mfB_1(\alpha))_+$, and by the definition of VaR, $\Delta_N(\alpha) > 0 \Leftrightarrow \alpha > \mbP[M_{N-1} < 0]$. To evaluate this probability, for $n\in\mbN_+$, let $A_n := \sum_{i=1}^n \bm{1}_{\{S_i > 0\}}$, i.e., the number of positive sums from $S_1$ to $S_n$. Since $Z_1,\cdots,Z_n$ are i.i.d. with a continuous symmetric distribution, from Sparre Andersen's theorem (Theorem 1 in \cite{Andersen1949saj}), we obtain that
    \[
    \mbP[M_n < 0] = \mbP[A_n = 0] = (-1)^n \binom{-\frac{1}{2}}{0}\binom{-\frac{1}{2}}{n} = \frac{(\frac{1}{2})(\frac{3}{2})\cdots(n-\frac{1}{2})}{n!} = \frac{1}{4^n}\binom{2n}{n},
    \]
    where the first equality uses the fact that $\mbP[S_i = 0] = 0$ for every $1\le i \le n$.
    Take $n = N-1$, we obtain $\alpha_0(N) := \mbP[M_{N-1} < 0] = \frac{1}{4^{N-1}}\binom{2N-2}{N-1}$. Clearly, $\alpha_0(2) = 1/2$ and $0 < \frac{\alpha_0(N+1)}{\alpha_0(N)} = \frac{2N-1}{2N} < 1$ for $N \ge 3$, so $\alpha_0(N)$ strictly decreases with $N$. We conclude that $\alpha_0(N) \le 1/2$ for all $N \ge 2$. Combining the two cases, we obtain the desired results. This completes the proof.
\end{proof}

From Section \ref{sec:VaROEDynamicRegime1}, the objective value under the optimal dynamic strategy can be recovered by applying the threshold stopping rule with stopping stage $1 + \min\{t_{\mfB_1(\alpha)},N-1\}$. So the problem remaining is how to compute $\mfB_1(\alpha)$ for any fixed $N$ and $\alpha$.

By Proposition \ref{prop:DynamicThresholdStrategy}, $\mfB_1(\alpha)=\VaR_\alpha[M_{N-1}]$, where
\[
    F_{N-1}(z):=\mbP[M_{N-1}<z] =\mbP[S_1<z,\cdots,S_{N-1}<z],\quad \forall z\in\mbR.
\]
The vector $(S_1,\cdots,S_{N-1})$ is centered Gaussian with covariance $\mathrm{Cov}(S_j,S_k)=\min\{j,k\}$. Thus, $\mfB_1(\alpha)$ can be computed by solving $F_{N-1}(z)=\alpha$ using numerical evaluation of the multivariate normal CDF and a root-finding method. However, evaluating this CDF becomes increasingly costly as $N$ grows.

For efficient computation, we exploit the identity $Y_m\overset{d}{=}M_m$ established in the proof of Proposition \ref{prop:DynamicThresholdStrategy}, where $Y_m$ is defined in \eqref{eq-defnSkRmYm}. Let $H_m(z):=\mbP[Y_m>z]$ for any $z\in\mbR$. Equation \eqref{eq-IntegralRecursion} yields
\[
    H_1(z)=\Phi(-z),\quad
    H_m(z)=\Phi(-z) + \int_0^\infty H_{m-1}(u)\varphi(z-u) du, \quad 2\le m \le N-1.
\]
We truncate the integral to $[0,B]$ and apply Gauss-Legendre quadrature with fixed nodes and weights \citep{Hildebrand1987book}. The recursion then reduces to repeated matrix-vector multiplication, with intermediate results reused across horizons and risk levels. Finally, we obtain $\mfB_1(\alpha)$ by solving $H_{N-1}(z)=1-\alpha$ using Brent's method. 

To control truncation error, observe that, for $z\ge0$, $1\le m\le N-2$,
\[
    H_m(z)\le\sum_{k=1}^m\mbP[S_k>z]
    =\sum_{k=1}^m\Phi\left(-\frac{z}{\sqrt{k}}\right)
    \le m\Phi\left(-\frac{z}{\sqrt m}\right).
\]
Consequently, for every $z\in\mbR$, $B > 0$, we have
\[
    0\le\int_B^\infty H_m(u)\varphi(z-u)du
    \le H_m(B) \le m\Phi\left(-\frac{B}{\sqrt m}\right).
\]
We set $B:=10\sqrt{N-1}$, so the omitted tail at step $m+1$ is bounded by $m\Phi(-10)$, where $\Phi(-10)\approx7.62\times10^{-24}$. This error is negligible for the horizons considered in the experiments. 

Figure \ref{fig:alpha0NDeltaNalpha} presents the numerical results. The left panel of Figure \ref{fig:alpha0NDeltaNalpha} shows that $\alpha_0(N)$ decreases from $0.5$ at $N=2$, first falls below $0.1$ at $N=33$, and reaches approximately $0.0566$ at $N=100$. 
The right panel of Figure \ref{fig:alpha0NDeltaNalpha} plots $\Delta_N(\alpha)$ for $\alpha=0.01,0.02,\cdots,0.99$ and $N\in\{3,5,10,33,100\}$. The results are computed using 1,000 quadrature nodes. As a numerical robustness check, we repeat the computation using 100 nodes. The maximum absolute difference between the two results is below $10^{-10}$ for $N\in\{3,5,10,33\}$ and below $1.6\times10^{-3}$ for $N=100$.
For each $N$, the gain is zero up to $\alpha_0(N)$, then increases to a maximum at $\alpha=0.5$ and decreases thereafter. We observe that larger horizons yield greater normalized gains, with diminishing increments. 

These results validate Propositions \ref{prop:DynamicThresholdStrategy} and \ref{prop:DynamicOutperformStatic}. In particular, $N=33$ already permits a strict dynamic advantage at $\alpha=0.1$, while increasing $N$ to $100$ raises the normalized gain by less than $0.07$. This illustrates diminishing gains from more frequent execution, which should be weighed against parameter restriction such as $N \le \frac{\kappa T}{2\eta}$, or any additional transaction costs or operational constraints in practice.

\begin{figure}[htbp]
    \centering
    \begin{minipage}{0.48\textwidth}
        \centering
        \includegraphics[width=\textwidth]{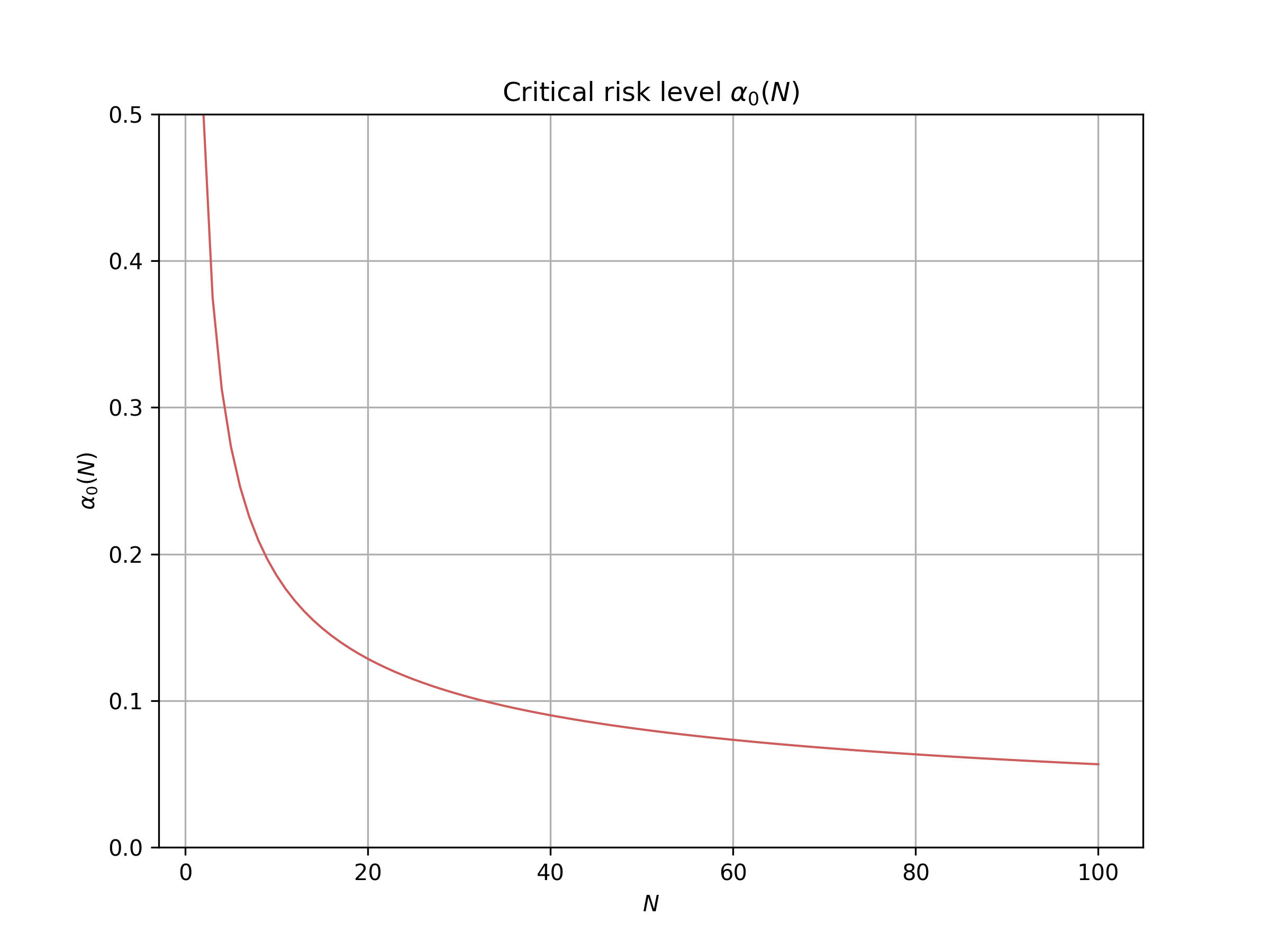}
    \end{minipage}
    \hfill
    \begin{minipage}{0.48\textwidth}
        \centering
        \includegraphics[width=\textwidth]{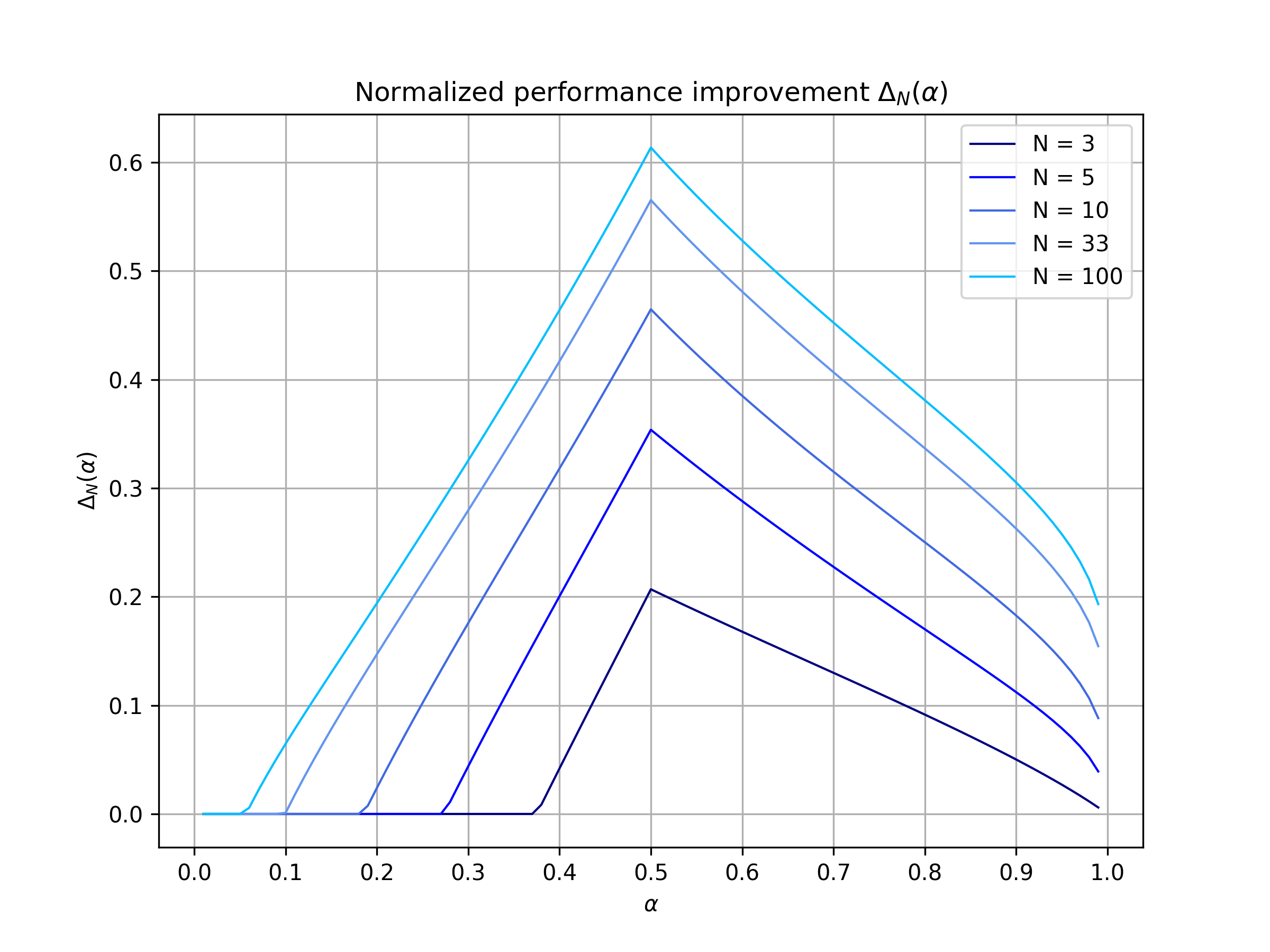}
    \end{minipage}
    \caption{Comparison of optimal dynamic and static strategies for Regime 1. Left: The change of the critical risk level $\alpha_0(N)$ with stage number $N$. Right: The change of the normalized $\alpha$-VaR performance improvement $\Delta_N(\alpha)$ with risk level $\alpha$ under different $N$.
    }
    \label{fig:alpha0NDeltaNalpha}
\end{figure}

\subsubsection{Regime 2}
We now test our result for Regime 2 numerically. We use the same transaction cost and market impact factors $\epsilon, \kappa, \eta$ as in Table 1 in \cite{Almgren2001jor}. Namely, we set 
\begin{equation}\label{eq-Regime2BaseParameters}
    S_0 = 50, \, \bar{x} = 10^6, \, T=5, \, \sigma = 0.95, \, \epsilon = 0.0625, \, \kappa = 2.5\times 10^{-7}, \, \eta = 2.5 \times 10^{-6}.
\end{equation}
In Regime 2, we require $N > \frac{\kappa T}{2\eta} = 0. 25$, so any $N \ge 2$ lies in this regime.

The main computational challenge in deriving the optimal dynamic strategy is evaluating $\mfD_k$, $\mfC_k$ and $\theta_{k+1}$ in \eqref{eq-RecursionRegime2}-\eqref{eq-RecursionRegime2thetak+1} for $k=1,\cdots,N-2$. To compute efficiently, we use backward induction on grids over inventory and risk levels. Namely, for each remaining inventory $y > 0$, we approximate
\[
    F_{k,y}(z) := \int_0^1 \Phi \left( \frac{z - \mfC_{k+1}(y,u)}{\sigma\sqrt{\tau}y} \right) du,\quad \forall z\in\mbR,
\]
using Gauss-Legendre quadrature and obtain $\mfD_k(y,\theta)$ by solving $F_{k,y}(z) = \theta$ with Brent's method \citep{Hildebrand1987book}. The quadrature nodes are included in the risk-level grid. We then maximize over feasible inventory-grid actions to compute $\mfC_k$, storing intermediate results for reuse. 

During the forward simulation, the risk level is updated by finding the smallest $\bar{\theta}$ satisfying $\mfC_{k+1}(y,\bar{\theta}) \ge \mfD_k(y,\theta) - \sigma\sqrt{\tau} y z_k$ using monotone interpolation between grid points. When the remaining inventory is zero, the continuation value is set to zero and no further risk-level update is required. We use Monte Carlo simulation to evaluate the resulting policies.

For the optimal static strategy, we maximize \eqref{eq-OEStaticStrategy} using SLSQP for each fixed risk level. As shown in Section \ref{sec:VaROEStatic}, in Regime 2, the objective is concave for $\alpha\le1/2$, and equal trading is optimal at $\alpha=1/2$. Hence, globally optimal solutions can be found. For $\alpha>1/2$, we retain the best solution across multiple initializations. Although global optimality is not guaranteed, different initializations yield the same objective value in our experiments.

\begin{figure}[htbp]
    \centering
    \includegraphics[width=\linewidth]{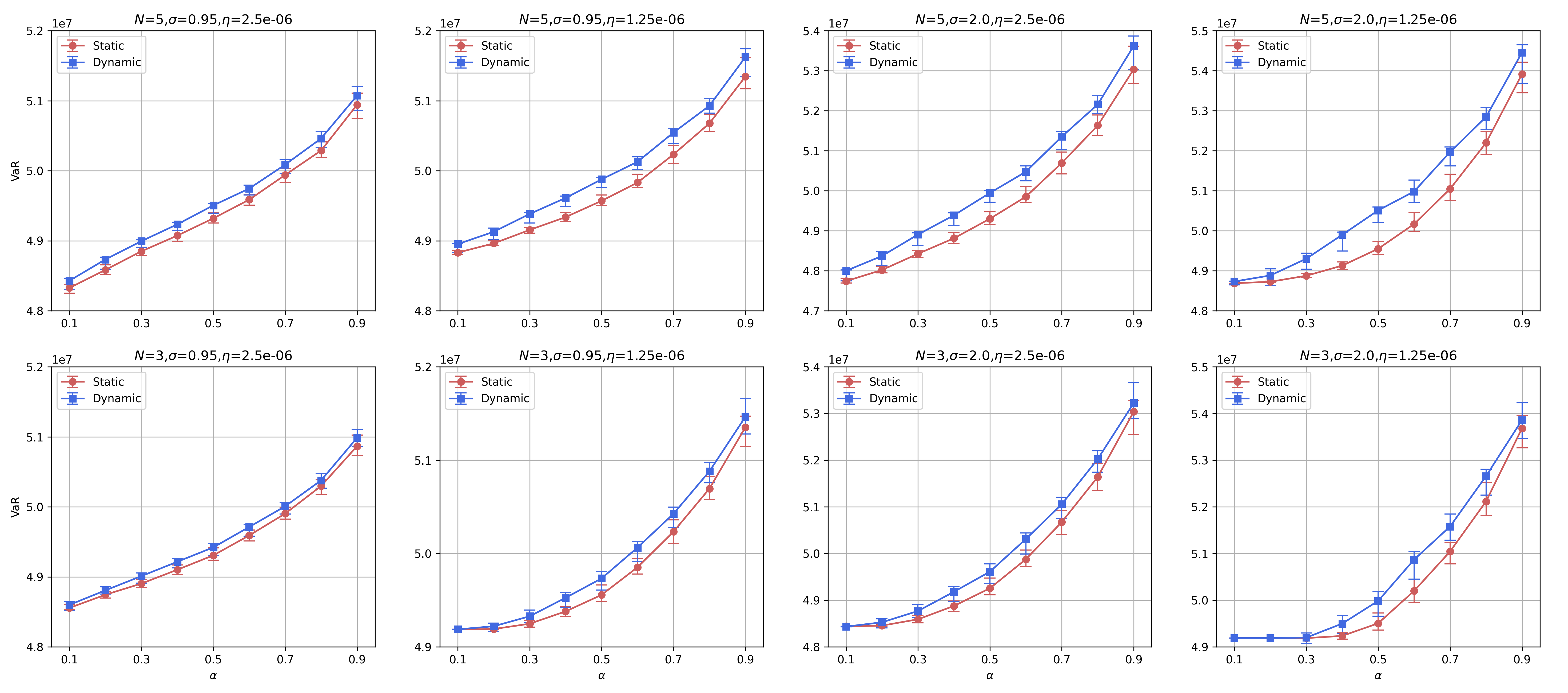}
    \caption{Comparison of static and dynamic strategies under different parameter settings for Regime 2}
    \label{fig:ComparisonRegime2M1000}
\end{figure}

We consider $\alpha\in\{0.1,\cdots,0.9\}$, $N\in\{3,5\}$. We refer to $N=5$ with \eqref{eq-Regime2BaseParameters} as the base case, which is the standard parameter setting in \cite{Almgren2001jor} Section 4. To facilitate parameter comparisons, we also consider $\sigma=2.0$ and $\eta=1.25\times10^{-6}$.
For each configuration, we simulate 1,000 paths using the optimal static and dynamic strategies under the common random numbers to estimate VaR and construct 95\% bootstrap confidence intervals. The grid discretization uses 1,000 inventory intervals and 100 Gauss–Legendre quadrature nodes, with the evaluated risk levels added to the risk grid. Although the analytical solution \eqref{eq-nN-1star}–\eqref{eq-mfCN-1} is available for stage $N-1$, we retain grid-based action optimization for consistency across stages. Finer grids can improve accuracy at greater computational cost. For $N=3$, comparing inventory discretizations with 100, 1,000, and 10,000 intervals yields a maximum difference below 2,000 in estimated VaR. 

Figure \ref{fig:ComparisonRegime2M1000} shows the empirical VaR of the optimal static and dynamic strategies under different parameter settings, evaluated using 1,000 simulated paths. The error bars over each point indicate 95\% bootstrap confidence intervals. 
We observe that the dynamic strategy outperforms its static counterpart across all parameter settings considered. The improvement is more pronounced for less extreme values of $\alpha$, particularly when $\alpha$ is close to $0.5$.
For small $\alpha$, increasing $N$ deteriorates the performance of both the static and dynamic strategies. This is because the execution price $\tilde{S}_k$ decreases as $N$ increases, which in turn reduces $C(S_0,\bar{x})$ and has a stronger adverse effect on revenue when $\alpha$ is small. When $\alpha$ is large, this effect is partially offset by the other terms in \eqref{eq-OEStaticObj} and \eqref{eq-RecursionRegime2vk}.
The maximum improvement of the dynamic strategy over the static strategy is $0.37\%$ in the base case. It decreases to $0.24\%$ when $N$ is changed to $3$, increases to $1.31\%$ when $\sigma$ is increased to $2.0$, and rises to $0.62\%$ when $\eta$ is changed to $1.25\times10^{-6}$. The largest improvement, $1.93\%$, occurs under the parameter setting $(N,\sigma,\eta)=(5,2.0,1.25\times10^{-6})$.

In summary, the numerical results demonstrate the computational feasibility and practical values of the VaR dynamic programming framework developed in Section \ref{sec:VaRDPFinite}. In particular, the proposed approaches provide tractable procedures for computing optimal dynamic strategies for the VaR optimal execution problem.

\section{Conclusions}\label{sec:Conclusion}
This work addresses VaR MDPs in continuous state and action spaces. We develop a VaR decomposition theorem and formulate a VaR dynamic programming framework with provable optimality for both finite- and infinite-horizon problems in general Borel spaces. Our framework is applied to an optimal execution problem, and numerical results validate our analysis while demonstrating the significant potential of the proposed theory.

Several research directions are worth exploring. First, efficient algorithms for solving VaR dynamic programming problems need to be developed. Second, for the infinite-horizon case, it is important to investigate conditions under which unbounded rewards can be accommodated. Finally, further applications should be explored in which VaR dynamic programming can be readily applied.

\section*{Acknowledgments}
Xian Chen is partially supported by the National Natural Science Foundation of China [Grant 12271454].
Erick Delage is partially supported by the Natural Sciences and Engineering Research Council of Canada [Grant RGPIN-2022-05261]. Weikai Wang is partially supported by IVADO and the Fonds de recherche du Québec [2007773]. 

\section*{Author Contributions}
Weikai Wang conceived, initiated, and coordinated the research project, contributed to the formulation of the theoretical results and the development, verification, and simplification of the proofs, designed and conducted the numerical experiments, and wrote the manuscript. Xian Chen developed the initial mathematical proofs and verified the theoretical results. Erick Delage contributed to the formulation of the theoretical results and supervised the research. All authors reviewed and revised the manuscript.

\appendix
\section*{Appendices}

\section{Auxiliary Lemmas}\label{sec:AuxiliaryLemmas}

\begin{lemma}\label{lemma:ConditionalVaRInequaltiy}
    Let $X$ and $Y$ be random variables on the probability space $(\Omega,\mcF,\mbP)$ taking values in $(\mbR,\mcB(\mbR))$ and $(E,\mcB(E))$, respectively. Let $\nu := \mbP\circ Y^{-1}$ denote the law of $Y$. Then, for every $\alpha \in [0,1]$,
    \[ \VaR_\alpha[X] \le \sup_{y \in \supp(\nu)} \VaR_\alpha[X|Y = y]. \]
\end{lemma}
\begin{proof}
    The claim trivially holds for $\alpha = 1$.
    For $\alpha \in [0,1)$, it suffices to show that 
    \[ b_\epsilon := \VaR_\alpha[X] - \epsilon \le \sup_{ y \in \supp(\nu)} \VaR_\alpha[X|Y =y],\quad \forall \epsilon > 0.\] 
    To see this, by the definition of VaR, we have $\mbP[X < b_\epsilon] \le \alpha$, and hence $\int_E \mu(\{X < b_\epsilon\},y) \nu(dy) \le \alpha$. This means that there exists some $y_\epsilon \in \supp(\nu)$ such that $\mu(\{X < b_\epsilon\},y_\epsilon) \le \alpha$. By \eqref{eq-ConditionalVaR}, this implies that $\VaR_\alpha[X| Y= y_\epsilon] \ge b_\epsilon$. This completes the proof.
\end{proof}

\begin{lemma}\label{lemma:UscNondecreasingisRightcontinuous}
    Let $g:\hat{\bfX}\to\bar{\mbR}$ be u.s.c. and $\alpha \mapsto g(x,\alpha)$ is non-decreasing for each $x\in\bfX$. Then $\alpha\mapsto g(x,\alpha)$ is right-continuous. 
\end{lemma}
\begin{proof}
    For any $(x,\alpha) \in \bfX\times[0,1)$, taking $\{\alpha_n\}_{n\in\mbN}\subset [0,1]$ such that $\alpha_n \downarrow \alpha$, by the monotonicity of $\alpha \mapsto g(x,\alpha)$, we have $g(x,\alpha_n) \ge g(x,\alpha)$. Thus $\liminf_{n\to\infty} g(x,\alpha_n) \ge g(x,\alpha)$. Moreover, since $g$ is u.s.c. on $\hat{\bfX}$ and $(x,\alpha_n) \to (x,\alpha)$, we have $\limsup_{n\to\infty} g(x,\alpha_n) \le g(x,\alpha)$. This implies the limit $\lim_{n\to\infty} g(x,\alpha_n)$ exists and equals $g(x,\alpha)$. Hence $\alpha\mapsto g(x,\alpha)$ is right-continuous.
\end{proof}

\begin{lemma}\label{lemma:NewPolicyMeasurable}
    Suppose Assumption \ref{assump:CompactContinuity} holds and fix $T \in \mbN_+$. For $t\in[T-1]$, let $\hat{f}_t \in \hat{\mbF}$ and $v_t$ defined as in \eqref{eq-Bellman-FH}. For $t \in [T-2]$, define
    \begin{equation}\label{eq-alphatSet}
        \mfa_{t+1}(x,\alpha,x') := \min\left\{ \bar{\alpha} \in [0,1]: \gamma v_{t+1}(x',\bar{\alpha}) \ge v_t(x,\alpha) - r_t(x,\hat{f}_t(x,\alpha)) \right\},
    \end{equation}
    for $\alpha \in [0,1)$ and set $\mfa_{t+1}(x,1,x') := 1$. Assume that the set in \eqref{eq-alphatSet} is nonempty. Then
    \begin{enumerate}[label = (\roman*)]
        \item $\mfa_{t+1}: \bfX\times[0,1]\times\bfX \to [0,1]$ is measurable.
        \item For every initial times $s\in [T-1]$ and initial risk level $\alpha \in [0,1]$, and $h_{s:k} \in \bfH_{s:k}$ for $k \ge s$,
        the recursively defined risk-level functions: 
        \begin{equation}\label{eq-rhoskalphaMeasurable}
        \begin{aligned}
            \rho_{s:s}^\alpha(h_{s:s}) &:= \alpha,\\
            \rho_{s:k+1}^\alpha(h_{s:k+1}) &:= \mfa_{k+1}(x_k,\rho_{s:k}^\alpha(h_{s:k}),x_{k+1}),\quad k=s,\cdots,T-2,
        \end{aligned}
        \end{equation}
        are measurable.
        \item The stochastic kernels
        \[
        \varpi_{s:k}^\alpha(\cdot|h_{s:k}) := \delta_{\hat{f}_k(x_k,\rho_{s:k}^\alpha(h_{s:k}))}(\cdot),\quad k=s,\cdots,T-1,
        \]
        define an admissible history-dependent policy
        \[
        \pi_{s:T-1}^\alpha := (\varpi_{s:s}^\alpha,\cdots,\varpi_{s:T-1}^\alpha) \in \Pi_{s:T-1}.
        \]
    \end{enumerate}
\end{lemma}
\begin{proof}
    (i) Fix $t\in[T-2]$. Since $\alpha \mapsto v_{t+1}(x',\alpha)$ is non-decreasing and u.s.c., the set
    \[
    D_{t+1}(x,\alpha,x') := \left\{ \bar{\alpha} \in [0,1]: \gamma v_{t+1}(x',\bar{\alpha}) \ge v_t(x,\alpha) - r_t(x,\hat{f}_t(x,\alpha))\right\},
    \]
    is a closed upper subset of $[0,1]$. By definition it is nonempty (since $\bar{\alpha}=1$ is always a solution), and hence compact, so its minimum is attained.

    For every $\beta \in [0,1)$, by definition, we have
    \begin{align*}
        &\{(x,\alpha,x'): \mfa_{t+1}(x,\alpha,x') \le \beta\} \\
        =& \{(x,\alpha,x'): \alpha \in [0,1), \gamma v_{t+1}(x',\beta) \ge v_t(x,\alpha) - r_t(x,\hat{f}_t(x,\alpha))\}.
    \end{align*}
    The set of the RHS is Borel since by Lemma \ref{lemma:vtvtuProperties} and Assumption \ref{assump:CompactContinuity}, $v_t, v_{t+1}, r_t$ and $\hat{f}_t$ are all measurable. Therefore, the function $\mfa_{t+1}$ is also measurable.

    (ii) We prove by induction. The initial function $\rho_{s:s}^\alpha\equiv\alpha$ is a constant and hence measurable. Suppose $\rho_{s:k}^\alpha$ is measurable for some $s < k \le T-2$. Since the coordinate mappings 
    \[
    h_{s:k+1}=(x_s,a_s,\cdots,x_k,a_k, x_{k+1}) \mapsto x_k,\quad h_{s:k+1} \to x_{k+1},
    \]
    are measurable, by (i), the mapping
    \[
    h_{s:k+1} \mapsto \mfa_{k+1}(x_k,\rho_{s:k}^\alpha(h_{s:k}),x_{k+1}),
    \]
    is also measurable. Hence, $\rho_{s:k+1}^\alpha$ is measurable. We conclude that the functions in \eqref{eq-rhoskalphaMeasurable} are measurable.

    (iii) For $h_{s:k} \in \bfH_{s:k}$, define $a_k^\alpha(h_{s:k}) := \hat{f}_k(x_k,\rho_{s:k}^\alpha(h_{s:k}))$. Clearly, $a_k^\alpha(h_{s:k}) \in \bfA(x_k)$. Then $\varpi_{s:k}^\alpha(\cdot|h_{s:k}) = \delta_{a_k^\alpha(h_{s:k})}(\cdot)$. Following (ii), and noting that $\hat{f}_k$ is measurable, we know that $h_{s:k} \mapsto \varpi_{s:k}^\alpha(\cdot|h_{s:k})$ is also measurable.
    Hence, the stochastic kernel $\varpi_{s:k}^\alpha$ is well-defined. Since $\varpi_{s:k}^\alpha$ assigns probability one to the feasible action set $\bfA(x_k)$, we conclude that $\pi_{s:T-1}^\alpha \in \Pi_{s:T-1}$.
\end{proof}

{\begin{small}
    \bibliography{reference.bib}

@inproceedings{Hau2023neurips,
  title={On dynamic programming decompositions of static risk measures in {Markov} decision processes},
  author={Hau, Jia Lin and Delage, Erick and Ghavamzadeh, Mohammad and Petrik, Marek},
  booktitle={Advances in Neural Information Processing Systems},
  volume={36},
  year={2023}
}

@article{LiXC2022opre,
  title={Quantile {Markov} decision processes},
  author={Li, Xiaocheng and Zhong, Huaiyang and Brandeau, Margaret L},
  journal={Operations Research},
  volume={70},
  number={3},
  pages={1428--1447},
  year={2022}
}

@inproceedings{Hau2025aistats,
    author = {J. L. Hau and E. Delage and E. Derman and M. Ghavamzadeh and M. Petrik},
    title = {Q-learning for Quantile {MDPs}: A Decomposition, Performance, and Convergence Analysis},
    booktitle = {The 28th International Conference on Artificial Intelligence and Statistics (AISTATS)},
    year = {2025}
}

@book{Follmer2016book,
  title={Stochastic Finance: An Introduction in Discrete Time, 4th Edition},
  author={F{\"o}llmer, Hans and Schied, Alexander},
  year={2016},
  publisher={Walter de Gruyter}
}

@book{Hernandez1996book,
  title={Discrete-time Markov Control Processes: Basic Optimality Criteria},
  author={Hern{\'a}ndez-Lerma, On{\'e}simo and Lasserre, Jean B},
  year={1996},
  publisher={Springer Science \& Business Media},
  address = {New York}
}

@book{Bauerle2011book,
  title={Markov Decision Processes with Applications to Finance},
  author={Nicole B\"{a}uerle and Ulrich Rieder},
  year={2011},
  publisher={Springer Berlin, Heidelberg}
}

@article{Bauerle2011mmor,
  title={Markov decision processes with average-value-at-risk criteria},
  author={Nicole B\"{a}uerle and Jonathan Ott},
  journal={Mathematical Methods of Operations Research},
  volume={74},
  pages={361--379},
  year={2011}
}

@article{Almgren2001jor,
  title={Optimal execution of portfolio transactions},
  author={Almgren, Robert and Chriss, Neil},
  journal={Journal of Risk},
  volume={3},
  pages={5--40},
  year={2001}
}

@article{Bauerle2024mmor,
  title={Markov decision processes with risk-sensitive criteria: an overview},
  author={Nicole B\"{a}uerle and Anna Jaśkiewicz},
  journal={Mathematical Methods of Operations Research},
  volume={99},
  number={1},
  pages={141--178},
  year={2024}
}

@article{Ruszczynski2010mp,
    author = {Andrzej Ruszczyński},
    title = {Risk-averse dynamic programming for {Markov} decision processes},
    journal = {Mathematical Programming},
    year = {2010},
    volume = {125},
    pages = {235--261}
}

@article{Jiang2018mor,
    author = {Daniel R. Jiang and Warren B. Powell},
    title = {Risk-averse approximate dynamic programming with quantile-based risk measures},
    journal = {Mathematics of Operations Research},
    year = {2018},
    volume = {43},
    number = {2},
    pages = {554--579}
}

@article{Shen2013sicon,
    author = {Yun Shen and Wilhelm Stannat and Klaus Obermayer},
    title = {Risk-sensitive {Markov} control processes},
    journal = {SIAM Journal on Control and Optimization},
    year = {2013},
    volume = {51},
    number = {5},
    pages = {3652–-3672}
}

@article{Cavus2014sicon,
    author = {\c{C}avu\c{s}, {\"O}zlem and Ruszczy\'{n}ski, Andrzej},
    title = {Risk-Averse Control of Undiscounted Transient Markov Models},
    journal = {SIAM Journal on Control and Optimization},
    year = {2014},
    volume = {52},
    number = {6},
    pages = {3935–-3966}
}

@article{Filar1995tac,
    author = {J. A. Filar and D. Krass and K. W. Ross},
    title = {Percentile Performance Criteria For Limiting Average {Markov} Decision Processes},
    journal = {IEEE Transactions on Automatic Control},
    year = {1995},
    volume = {40},
    number = {1},
    pages = {2--10}
}

@article{Pflug2016mor,
    author = {Georg Ch. Pflug and Alois Pichler},
    title = {Time-Consistent Decisions and Temporal Decomposition of Coherent Risk Functionals},
    journal = {Mathematics of Operations Research},
    year = {2016},
    volume = {41},
    number = {2},
    pages = {682--699}
}

@article{Bauerle2026ejor,
    author = {Nicole B\"{a}uerle and Tamara G\"{o}ll and Anna Ja\'{s}kiewicz},
    title = {Yet Another Distributional {Bellman} Equation},
    journal = {European Journal of Operational Research},
    year = {2026},
    volume = {334},
    number = {2},
    pages = {392--403}
}

@article{Howard1972mnsc,
  title={Risk-sensitive {Markov} decision processes},
  author={Howard, R. A. and Matheson, J. E.},
  journal={Management Science},
  volume={18},
  number={7},
  pages={356--369},
  year={1972}
}

@article{Bauerle2022ejor,
    author = {Nicole B\"{a}uerle and Alexander Glauner},
    title = {Markov decision processes with recursive risk measures},
    journal = {European Journal of Operational Research},
    year = {2022},
    volume = {296},
    pages = {953--966}
}

@article{JinYS2017qf,
    author = {Yasong Jin},
    title = {Optimal execution strategy and liquidity adjusted value-at-risk},
    journal = {Quantitative Finance},
    year = {2017},
    volume = {17},
    number = {8},
    pages = {1147--1157}
}

@techreport{BCBS2016,
  author       = {{Basel Committee on Banking Supervision (BCBS)}},
  title        = {Minimum Capital Requirements for Market Risk—January 2016},
  institution  = {Basel Committee on Banking Supervision, Bank for International Settlements},
  type         = {Standards},
  location     = {Basel, Switzerland},
  year         = {2016}
}

@article{Emmer2015jor,
    author = {Susanne Emmer and M. Kratz and Dirk Tasche},
    title = {What is the Best Risk Measure in Practice? A Comparison of Standard Measures},
    journal = {Journal of Risk},
    year = {2015},
    volume = {18},
    number = {2},
    pages = {31--60}
}

@article{deCastro2019econometrica,
    author = {de Castro, Luciano and Antonio F. Galvao},
    title = {Dynamic Quantile Models of Rational Behavior},
    journal = {Econometrica},
    year = {2019},
    volume = {87},
    number = {6},
    pages = {1893--1939}
}

@article{deCastro2022geb,
    author = {de Castro, Luciano and Antonio F. Galvao and Charles N. Noussair and Liang Qiao},
    title = {Do people maximize quantiles?},
    journal = {Games and Economic Behavior},
    year = {2022},
    volume = {132},
    pages = {22--40}
}

@article{HeXD2025mnsc,
    author = {Xue Dong He and Zhaoli Jiang and Steven Kou},
    title = {Dynamic Portfolio Selection Under Quantile Maximization},
    journal = {Management Science},
    year = {2025},
    note = {Article in Advance}
}

@article{Artzner2001mf,
    author = {Philippe Artzner and Freddy Delbaen and Jean-Marc Eber and David Heath},
    title = {Coherent measures of risk},
    journal = {Mathematical Finance},
    year = {2001},
    volume = {9},
    number = {3},
    pages = {203--228}
}

@article{Xia2025arXiv,
    author = {Xia, Li and Pan, Jinyan},
    title = {Markov decision processes with value-at-risk criterion},
    journal = {arXiv preprint, arXiv:2507.22355},
    year = {2025}
}

@inproceedings{LuoY2026icml,
    author = {Yu Luo and Erick Delage},
    title = {Boosting {CVaR} Policy Optimization with Quantile Gradients},
    booktitle = {Proceedings of the 43rd International Conference on Machine Learning},
    year = {2026}
}

@article{Bauerle2014mor,
    author = {Nicole B\"{a}uerle and Ulrich Rieder},
    title = {More risk-sensitive {Markov} decision processes},
    journal = {Mathematics of Operations Research},
    year = {2014},
    volume = {39},
    number = {1},
    pages = {105--120}
}

@inproceedings{Hau2023aistats,
    author = {J. L. Hau and M. Petrik and M. Ghavamzadeh},
    title = {Entropic risk optimization in discounted {MDPs}},
    booktitle = {Proceedings of the 26th International Conference on Artificial Intelligence and Statistics (AISTATS)},
    year = {2023}
}

@article{Biswas2023arc,
    author = {A. Biswas and V. S. Borkar},
    title = {Ergodic risk-sensitive control -- A survey},
    journal = {Annual Reviews in Control},
    year = {2023},
    volume = {55},
    pages = {118--141}
}

@article{Bauerle2021mmor,
    author = {Nicole B\"{a}uerle and Alexander Glauner},
    title = {Minimizing spectral risk measures applied to {Markov} decision processes},
    journal = {Mathematical Methods of Operations Research},
    year = {2021},
    volume = {94},
    pages = {35--69}
}

@article{Rudloff2014ejor,
    author = {Birgit Rudloff and Alexandre Street and Davi M. Valladão},
    title = {Time consistency and risk averse dynamic decision models: Definition, interpretation and practical consequences},
    journal = {European Journal of Operational Research},
    year = {2014},
    volume = {234},
    pages = {743--750}
}

@article{Kreps1977mor,
    author = {D. M. Kreps},
    title = {Decision problems with expected utility criteria, {I}: Upper and lower convergent utility},
    journal = {Mathematics of Operations Research},
    year = {1977},
    volume = {2},
    number = {1},
    pages = {45--53}
}

@book{Billingsley1999book,
    author = {Patrick Billingsley},
    title = {Convergence of Probability Measures, Second Edition},
    publisher = {John Wiley \& Sons},
    year = {1999}
}

@book{Kallenberg2021book,
    author = {Olav Kallenberg},
    title = {Foundations of Modern Probability, Third Edition},
    publisher = {Springer},
    year = {2021}
}

@article{Andersen1949saj,
    author = {Erick Sparre Andersen},
    title = {On the number of positive sums of random variables},
    journal = {Scandinavian Actuarial Journal},
    year = {1949},
    volume = {1},
    pages = {27--36}
}

@book{Hildebrand1987book,
    author = {F. B. Hildebrand},
    title = {Introduction to Numerical Analysis: Second Edition},
    publisher = {Dover Publications},
    year = {1987}
}

@article{Dentcheva2020mp,
    author = {D. Dentcheva and A. Ruszczyński},
    title = {Risk forms: representation, disintegration, and application to partially observable two-stage systems},
    journal = {Mathematical Programming},
    year = {2020},
    volume = {181},
    pages = {297--317}
}

@article{Fan2022mp,
    author = {J. Fan and A. Ruszczyński},
    title = {Process-based risk measures and risk-averse control of discrete-time systems},
    journal = {Mathematical Programming},
    year = {2022},
    volume = {191},
    pages = {113--140}
}

@article{MinSK2022arXiv,
    author = {Seungki Min and Ciamac C. Moallemi and Costis Maglaras},
    title = {Risk-Sensitive Optimal Execution via a Conditional Value-at-Risk Objective},
    journal = {arXiv preprint arXiv:2201.11962},
    year = {2022}
}
\end{small}}

\end{document}